\documentclass[11pt]{article}
\usepackage{verbatim}
\usepackage[total={6.5in, 8.75in}, top=1.2in, left=0.9in,includefoot]{geometry}

\usepackage{tikz}
\usepackage{amsmath, amsfonts, amsthm}
\usepackage{amssymb}
\newtheorem{theorem}{Theorem}[section]
\newtheorem{lemma}[theorem]{Lemma}
\newtheorem{proposition}[theorem]{Proposition}

\title{On The Characteristic Polynomial of Frobenius of Supersingular Abelian Varieties Of Dimension up to 7 over Finite Fields}
\author{Vijaykumar Singh, Alexey Zatysev, Gary McGuire}
\begin{document}
\maketitle

\begin{abstract}  In this article, we derive the list of  the characteristic polynomials of the Frobenius endomorphism of  simple supersingular abelian varieties of dimension $1,~2,~3,~4,~5,~6,~7$ over $\mathbb{F}_q$ where $q=p^n$,~ $n$ odd. 
\end{abstract}
\section{Introduction}
Supersingular abelian varieties have applications in cryptography  and coding theory and related areas. Identity based ecryption and computation of weights of some Reed Muller codes are some of them. The important (isogeny) invariant which carries most of the information about supersingular curves is the characteristic polynomial of the Frobenius endomorphism. Here we give a list of such polynomials up to dimension $7$ over $\mathbb{F}_q$ where $q=p^n$,~ $n$ odd. We also give the procedure, which extends  to all dimensions.

Let $A$ be an abelian variety of dimension $g$ over $\mathbb{F}_q$ where $q =p^n$. For $l \neq p$, the characteristic polynomial of Frobenius endomorphism $\alpha$ is defined as,
\begin{displaymath}
P_A(X):=\det(\alpha-XId|V_l(A)).
\end{displaymath}
The above definition is independent of choice of $l$. The coefficients of  $P_A(X)$ are in $\mathbb{Z}$. In fact,  
\begin{displaymath}
P_A(X) = X^{2g} + a_{1}X^{2g-1} + \cdots+ a_gX^g +qa_{g-1}X^{g-1} \cdots + q^g.
\end{displaymath}
An abelian variety $A$ is $k$-simple if it is not isogenous to a product of abelian varieties of lower dimensions over $k$. In that case $P_A(X)$ is either irreducible over $\mathbb{Z}$ or $P(X)=h(X)^e$ where $h(X)\in \mathbb{Z}[X]$ is an irreducible over $\mathbb{Z}$ , see \cite{Water-Milne}.
We have the following result from Tate \cite{Tate}.
\begin{theorem} \label{Tate} If $A$ and $B$ are the abelian varieties defined over $\mathbb{F}_q$. Then $A$ is $\mathbb{F}_q$-isogenous to abelian subvariety of $B$ if and only if $P_A(X)$ divides $P_B(X)$ over $\mathbb{Q}[X]$. In particular, $P_A(X) = P_B(X)$ if and only if $A$ and $B$ are $\mathbb{F}_q$-isogenous.
\end{theorem}

We can factor $P_A(X)$ over the complex numbers as $P_A(X)= \prod_{i=1}^{2g}(X-\alpha_i)$ where the $\alpha_i$ are algebraic integers.
 An algebraic integer $\pi \in \mathbb{C}$ is called a Weil-$q$-number if for every embedding $\sigma:\mathbb{Q}(\pi)\hookrightarrow \mathbb{C}$, $\sigma(\pi)=\sqrt{q}$.
 Let $W(q)$ be the set of a Weil-$q$-numbers in $\mathbb{C}$. Two elements $\pi$ and $\pi^{'}$  are conjugates ($\pi \sim \pi^{'})$ if they have same minimal polynomial over $\mathbb{Q}$. In fact we have following  one to one correspondence.
\begin{theorem}(Honda-Tate theorem)
 The map $A \rightarrow \pi_A$ defines
a bijection
\begin{displaymath} 
\{simple~abelian~varieties/ \mathbb{F}_q\}/(isogeny)\mapsto W(q)/ (conjugacy)
\end{displaymath}
\end{theorem}
An elliptic curve $E$ over  $\mathbb{F}_q$ is supersingular if $E({ \mathbb{\bar{F}}_q})$ has no points of order $p$.
 An abelian variety A over $\mathbb{F}_q$  is called supersingular if A is isogenous over ${\mathbb{\bar{F}}_{q}}$ to a product of supersingular elliptic curves.
\begin{theorem}(Manin-Oort)
 $A/\mathbb{F}_q$ is supersingular $\iff$ $\pi_A=\sqrt{q}\zeta$, where $\zeta$ is some root of unity.  
\end{theorem}
We call a Weil-$q$ number $\pi$,  a supersingular Weil -$q$-number if $\pi_A=\sqrt{q}\zeta$, where $\zeta$ is some root of unity.  

Our approach is first we will compute all  supersingular  Weil -$q$-numbers (polynomials). Then we will find the dimension of the corresponding abelian varieties to those polynomials. 

\section{Supersingular Weil Polynomial}
In this section we give results and procedure for computing the supersingular Weil Polynomial. Let $P(X)$ be a Weil polynomial given by 
\begin{displaymath}
P(X) = X^{2g} + a_{1}X^{2g-1} + \cdots+ a_gX^g +qa_{g-1}X^{g-1} \cdots + q^g.
\end{displaymath}
We have following cases.
\subsection{$P(X)$ Irreducible}
\begin{proposition}\label{S2}
If $P(X)$ is irreducible with $a_{2i+1}\neq 0$ for some $i$, then $a_i$'s satisfy\\
\begin{displaymath}
H(t):=(t^{2g}+\frac{a_{2}t^{2g-2}}{q}+\cdots+\frac{a_{2i}t^{2g-2i}}{q^i}+\cdots+1)^2- \frac{1}{q}(a_1t^{2g-1}+\frac{a_3t^{2g-3}}{q}+\cdots+a_1t)^2=0.
\end{displaymath}
\end{proposition}
\begin{proof}
By dividing  $P(X)$ by $q^g$ we get,
\begin{displaymath}
\frac{P(X)}{q^g}= \frac{X^{2g}}{q^g}+a_1\frac{X^{2g-1}}{q^g}+ \cdots +a_g \frac{X^g}{q^g}+qa_{2g-1}\frac{{X^{2g-1}}}{q^g}+\cdots+1=0.
\end{displaymath}
Doing the transformation $\frac{X}{\sqrt{q}} \rightarrow t$ over $\mathbb{Q}(\sqrt{q})$  and rearranging we get\\
\begin{displaymath}
G(t):=(t^{2g}+\frac{a_{2}t^{2g-2}}{q}+\cdots+\frac{a_{2i}t^{2g-2i}}{q^i}+\cdots+1)+\frac{1}{\sqrt{q}}(a_1t^{2g-1}+\frac{a_3t^{2g-3}}{q}+\cdots+a_1t)=0.
\end{displaymath}
Since $q=p^n$, $n$ odd,  $Gal(\mathbb{Q}({\sqrt{q}}) / \mathbb{Q})=\{1,\sigma\}$ where $\sigma(\sqrt{q})=-\sqrt{q}$. Therefore

\begin{displaymath}
H(t):=G(t)G(t)^{\sigma}=(t^{2g}+\frac{a_{2}t^{2g-2}}{q}+\cdots+\frac{a_{2i}t^{2g-2i}}{q^i}+\cdots+1)^2-\frac{1}{q}(a_1t^{2g-1}+\frac{a_3t^{2g-3}}{q}+\cdots+a_1t)^2=0.
\end{displaymath}
\end{proof}

The polynomial $H(t)$ defined above has only even powers of $t$.
\begin{proposition} \label{Red}
 Let  $G(t)\in \mathbb{Q}({\sqrt{q}})[t]$ (as defined above) be reducible over $\mathbb{Q}({\sqrt{q}})$. Then  $G(t)=F_{1}(t)F_{2}(t)$  where
$F_{1}, F_{2}$ are irreducible polynomial  in  $\mathbb{Q}({\sqrt{q}})[t]$ with $\deg F_{1},\deg F_{2}=g $.
\end{proposition}
\begin{proof}
 Let $G(t)=F_{1}(t)F_{2}(t)\ldots F_k(t)$  where $F_i$ are irreducible over  $\mathbb{Q}({\sqrt{q}})[t]$. 
Then $G(\sqrt{q}t)=F_{1}(\sqrt{q}t)F_{2}(\sqrt{q}t)\ldots F_k(\sqrt{q}t)$.
 Let $\sigma \in Gal(\mathbb{Q}({\sqrt{q}}) \setminus \mathbb{Q})$. Then,
 \begin{displaymath}
\frac{P(t)^2}{q^{2g}}=G(\sqrt{q}t)G(\sqrt{q}t)^{\sigma}=F_{1}(\sqrt{q}t)F_{1}(\sqrt{q}t)^{\sigma}F_{2}(\sqrt{q}t)F_{2}(\sqrt{q}t)^{\sigma}\ldots F_{k}(\sqrt{q}t)F_{k}(\sqrt{q}t)^{\sigma}.
\end{displaymath}
If $\deg F_i < g$ then $\deg F_i\deg F_{i}^{\sigma} < 2g$. But $F_{i}F_{i}^{\sigma} \in \mathbb{Q}[t]$ This implies there is polynomial of degree less than $2g$ over $\mathbb{Q}$ with $\alpha_i$, which contradicts the minimality of $P(t)$. Hence
$\deg F_i = g$ whence the theorem follows.
\end{proof}

\begin{theorem}\label{Odd}Assume $a_{2i+1}\neq 0$.
\begin{enumerate}
\item If  $G(t)$ is irreducible over $\mathbb{Q}(\sqrt{q})$ then $H(t)$ is an irreducible cyclotomic factor of $t^m-1$ of degree $4g$ where $\phi(m)=4g$ over $\mathbb{Q}$.
\item If $G(t)$ is reducible over  $\mathbb{Q}(\sqrt{q})$ then $G(t)=F_{1}(t)F_{2}(t)$ where  at least one of $F_{1},~F_{2} \in \mathbb{Q}(\sqrt{q})[t] \setminus \mathbb{Q}[t]$. If $F_i \in \mathbb{Q}[t]$ then it is an irreducible cyclotomic factor of  $t^m-1$ of degree $g$   where $\phi(m)=g$, else  $F_{i}F_{i}^{\sigma} \in \mathbb{Q}[t]$ is an irreducible cyclotomic factor  of $t^m-1$ of degree $2g$ where $\phi(m)=2g$.
\end{enumerate}
\begin{proof}
\begin{enumerate}
\item From the construction, the roots of $H(t)$ are  $\frac{\alpha_i}{\sqrt{q}}$ which are also the roots of unity. But $H(t)$ is irreducible, hence the assertion follows. 
\item Since $a_{2i+1} \neq 0$,  $G(t)\in  \mathbb{Q}(\sqrt{q})[t] \setminus \mathbb{Q}[t]$. Therefore $G=F_1F_2$ (by \ref{Red})implies at least one of them (say $F_1$) is in $\mathbb{Q}(\sqrt{q})[t]\setminus \mathbb{Q}[t]$, in which case $F_1F_{1}^{\sigma}$ is irreducible cyclotomic factor $t^m-1$ of degree $2g$ where $\phi(m)=2g$.
\end{enumerate}
\end{proof}
\end{theorem}
\begin{proposition}\label{Even}
If $a_{2i+1}=0$ for all $i$ then
\begin{enumerate}
\item If $G(t)$ is irreducible over $\mathbb{Q}$, it is an irreducible cyclotomic factor of $t^m-1$  of degree $g$ where $\phi(m)=2g$.
\item If $G(t)$ is reducible over  $\mathbb{Q}$, then $G(t)=F_1F_2$ where $F_i$ are irreducible cyclotomic factor  of $t^m-1$ of degree $g$ where $\phi(m)=g$.
\end{enumerate}
\end{proposition}
\begin{proof}
The proof is similiar to the proof of the previous theorem.
\end{proof}
\subsection{$P(X)$ Reducible}
If $P(X)$ is reducible then $P(X)=h(X)^e$ which implies $e|2g$. Also if $h(X)=\sum h_iX^i$ then $h_{0}^e=q^g=p^{ng}$ which implies $e|ng$.  But since $n$ is odd we have $e|g$. Therefore $\frac{P(X)}{g}=(\frac{h(X)}{q^{\frac{g}{e}}})^e$. Doing the transformation $\frac{X}{\sqrt{q}} \rightarrow t$ we get $G(t)$ as in proposition \ref{S2}. The same results on $G(t)$ from the above section holds depending on whether $G(t)$ is irreducible or reducible except that $g$ is replaced by $\frac{g}{e}$.\\

Using the above propositions we have the following procedure to derive the Weil polynomial.\\
\subsection{Procedure}
\begin{enumerate}
\item Suppose $P(X)$ is irreducible.
\begin{enumerate}
\item If $a_{2i+1}\neq 0$ for some $i$.
\begin{enumerate}
\item If $G(t)$ is irreducible, then as  in proposition \ref{Odd} find  appropriate $H(t)$ and solve for the $a_i$'s. 
\item If $G(t)$ is reducible then use proposition \ref{Odd}  to find appropriate  $H(t)=G(t)G(t)^{\sigma}$ and solve for the $a_i$'s .
\end{enumerate}

\item If  $a_{2i+1}= 0$ for all $i$.
 \begin{enumerate}
\item If $G(t)$ is irreducible over $\mathbb{Q}$ then use proposition \ref{Even} to find appropriate cyclotomic factors and compare with $G(t)$ to solve for $a_i$'s.
\item If $G(t)$ is reducible over $\mathbb{Q}$ then use proposition \ref{Even} to compute $G(t)=F_1(t)F_2(t)$ and compare coefficients to solve for $a_i$'s.
\end{enumerate}
\end{enumerate}
\item Suppose $P(X)$ is reducible i.e;  $P(X)=h(X)^e$ with $e|g$.  Then for  $\frac{h(X)}{q^{\frac{g}{e}}}$, doing the transformation $\frac{X}{\sqrt{q}}\rightarrow t$, we get a new polynomial $G(t)$  whose roots are roots of unity, and using above tools we can find $h(X)$.
\end{enumerate}

 Often in the above procedure, solutions to  the $a_i$'s are given by the roots of some polynomials $f(z,q)$ over $\mathbb{Z}$. We use following test to check if $f(z,q)$ has  an integer solution.

\begin{lemma}(Mod $3, 5$  test)
Let $f(z)$ be a monic polynomial of degree $d$ with coefficients in $\mathbb{Z}[q]$ where $q =p^n$ such that if we put
\begin{enumerate}
 \item $q=1,2  $ and $f(z)\mod 3$ has no solutions,
 \item $q=-1,1 $ and $f(z)\mod 5$ has no solutions.
\end{enumerate}
then $f(z)$ has no solutions in $\mathbb{Z}$ for any $q$.
\end{lemma}
\begin{proof}
If $a$ is an integer solution then $z-a\mid f(z)$.
Let $q$ be power of prime. If $q\equiv  1 \mod 3, ~q\equiv 2\mod 3$ then  $a \mod 3$ is a  solution.
 If $q = 0 \mod 3$ then  $q=-1$ or $1~~  \mod 5$ and $a \mod 5$ is the solution for $f(z) \mod 5$.
\end{proof}
In the next section we will calculate the dimension of the corresponding abelian variety to the supersingular Weil polynomial.
\section{Dimension}
Let $ \pi \in \bar{\mathbb{Q}}$ be the Weil number and  $P(X)$ be minimal Weil polynomial, with $\pi$ obtained by method above .We have the following theorem to calculate the dimension.
\begin{theorem}\label{dim} Let A be a simple abelian variety over $k = \mathbb{F}_q$ ,then
\begin{enumerate}
\item $End_k (A)\otimes \mathbb{Q}$ is a division algebra with center $\mathbb{Q}(\pi_A)$ and
                 
\begin{displaymath}
2dim A = [End_k (A)\otimes \mathbb{Q}:\mathbb{Q}(\pi_A)]^\frac{1}{2} [\mathbb{Q}(\pi_A) : \mathbb{Q}] .
\end{displaymath}
\item The division algebra $End_k (A)\otimes \mathbb{Q}$ over $\mathbb{Q}(\pi_A)$ has the following splitting
behaviour\\
\begin{enumerate}
\item it splits at each divisor $\mathfrak{l}$ of $l$ in $\mathbb{Q}(\pi_A)$, if $l\neq p$,
\item the invariants at the divisors  $\mathfrak{p}$ of p in  $\mathbb{Q}(\pi_A)$ can be evaluated with
\begin{displaymath}                                       
  inv_{\mathfrak{p}} (End_k (A)\otimes \mathbb{Q}) \equiv \frac{v_\mathfrak{p}(\pi_A)}{v_\mathfrak{p}(q)} [\mathbb{Q}(\pi_A)_\mathfrak{p} : \mathbb{Q}_p ] \mod \mathbb{Z} ,
 \end{displaymath}                                       
\item it does not split at the real places of $\mathbb{Q}(\pi_A)$.
\end{enumerate}
\end{enumerate}
\end{theorem}

The invariants of $End_k (A)\otimes \mathbb{Q}$ lie in $\mathbb{Q}/\mathbb{Z}$. They can be evaluated from the
minimal polynomial $P(X)$ of $\pi_A$ as follows.
The only real Weil numbers are $q^{1/2}$ and -$q^{1/2}$ , so there are hardly
any real places of $\mathbb{Q}(\pi_A)$. We consider the polynomial $P(X)$ in $\mathbb{Q}_p [X]$, i.e., over the p-adic numbers. Let
                                      
\begin{displaymath}
P (X) = \displaystyle\prod_{i}f_i (X)
\end{displaymath}                                      
be the decomposition in irreducible factors in $\mathbb{Q}_p [X]$. The factors $f_i(X)$ correspond uniquely to the divisors $\mathfrak{p}_i$ of $p$ in $\mathbb{Q}(\pi_A )$. 
So to get the invariants we have the factor $P(X)$ over $\mathbb{Q}_p$. In fact,
\begin{displaymath}
inv_{\mathfrak{p}_i} (End_k (A)\otimes \mathbb{Q}) \equiv \frac{v_p (f_i (0))}{ v_p (q)}  \mod \mathbb{Z}.
\end{displaymath}
We use the invariants in order to evaluate the dimension of A as follows.
 The number $[End_k (A)\otimes \mathbb{Q}:\mathbb{Q}(\pi_A)]^\frac{1}{2}$ is equal to the order of
$End_k (A)\otimes \mathbb{Q}$ in the Brauer group of $\mathbb{Q}(\pi_A)$ see  theorem 18.6, \cite{Pierce},
which in turn is equal to the least common multiple of the orders of all the local invariants in $\mathbb{Q}/\mathbb{Z}$ see theorem 18.5, \cite{Pierce}. This along with theorem \ref{dim} gives the dimension of $A$.

Hence the main problem in computing dimension is the factorization the $P(X)$ over p-adic numbers for which we will use the following result.
\begin{theorem}\label{padics}
Let $\Phi_n$ be nth cyclotomic polynomial in $\mathbb{Q}_p$. Then 
\begin{enumerate}
\item If $n=p^n$, then  $\Phi_n$ remains irreducible in $\mathbb{Q}_p$.
\item If $(n,p)=1$, then $\Phi_n=f_i\ldots f_r$, with $rf=\phi(n),~ \deg f_i=f$  for each $i$,  where $f$ is  the multiplicative order of $p \mod n$.
\end{enumerate}
\end{theorem}
\begin{proof}
See  \cite{Serre}, chapter IV.4.
\end{proof}
\section{Dimension 1}
The characteristic polynomial of Frobenius of a  dimension $1$ abelian variety is given by $P(X)=X^2+a_{1}X+q$.\\
\emph{Case $a_1=0$}: If $a_1=0$ then P(X)=$X^2+q$ is Frobenius of supersingular abelian variety for all $p$.\\
\emph{Case $a_1 \neq 0$}: If $a_1 \neq 0$ doing transformation $x=\frac{X}{\sqrt{q}}$ we get,\\
\emph{Case 1:} If $G(x)$ is irreducible then by proposition \ref{S2}
$H(x):=G(x)G(x)^{\sigma}=x^4+1+(2-\frac{a_{1}^{2}}{q})x^2$ with $x$ as $m^{th}$ root of unity where $\phi(m)=4.g=4$ which implies $m =\{5,~8,~10,~12$\}.\\
We have,
\begin{enumerate}
\item $x^5-1= \left( x-1 \right) \left( {x}^{4}+{x}^{3}+{x}^{2}+x +1\right)$ .
\item $x^8-1=  \left( x-1 \right) \left( 1+x \right)  \left( 1+{x}^{2} \right)\left( {x}^{4}+1 \right)$ .
\item $x^{10}-1=\left( x-1 \right)\left( 1+x \right)  \left( 1+{x}^{4}+{x}^{3}+{x}^{2}+x \right)\left(1-x+{x}^{2}-{x}^{3}+{x}^{4} \right)$.
\item $x^{12}-1=\left( x-1 \right) \left( 1+x \right)  \left( 1+{x}^{2}+x \right)\left( 1-x+{x}^{2} \right) \left( 1+{x}^{2} \right) \left( {x}^{4}-{x}^{2}+1 \right)$.
\end{enumerate}
Since $H(x)$ has only even powers of $x$ the only possibilities  are  $H(x)= x^{4}+1 $ or $H(x)=x^{4}-x^{2}+1$. Comparing the coefficients we get $a_{1}=\pm \sqrt{2q}$ or $\pm \sqrt{3q}$ which is an integer if and only if  $q$ is an odd power of $2$ and $3$ respectively. Therefore $a_{1}=\pm \sqrt{2q}$ or $\pm \sqrt{3q}$.\\
\emph{Case 2.} $G(x)$ is reducible. Then  by \ref{Red} $G(x)=F_1(x).F_2(x)$.
The roots of $F_i$  are mth root of unity  where $\phi(m)=2g=2$  if $F_i \in \mathbb{Q}({\sqrt{q}})[x]\setminus \mathbb{Q}[x]$ else $\phi(m)=g=1$. 
Then $H(x)=F_1F_{1}^{\sigma}F_2F_{2}^{\sigma}$ has no possibility with only even degree terms.\\
\emph{Case $P(X)$ is reducible}: If $P(X)$ is reducible then $P(X)=(X+a)^2 \neq 0$, then $a=\sqrt{q}$ has no integer solution as  $q$ is a odd power of prime. 
Therefore we have
\begin{theorem}\label{D1}  The characteristic polynomial of  a simple supersingular abelian variety of the dimension $1$ over $\mathbb{F}_q$ ($q=p^n$, $n$ odd) is one the following
\begin{enumerate}
\item $p=2:X^2 \pm\sqrt{2q}X+q$
\item $p=3:X^2\pm \sqrt{3q}X+q$
\item $X^2+q$
\end{enumerate}
\end{theorem}
In fact all of them appear, see \cite{Waterhouse}.
\section{Dimension 2}
 The characteristic polynomial of  a supersingular abelian variety of dimension $2$ is given by $P(X)=X^4+a_{1}X+a_{2}X^2+qa_1X+q^2$. We have following cases.
\subsection{Case $a_{2i+1} \neq 0$}\label{G2}
Let $P(X)$ be irreducible with $a_{2i+1} \neq 0$.
 On doing transformation $x=\frac{X}{\sqrt{q}}$ we get\\
$H(x)=\left( x^{4}+\frac {a_{2}x^{2}}{q}+1 \right) ^{2}-\frac{1}{q} \left( a_1x^{3}+a_1x \right) ^{2}$ whose roots are  mth roots of unity where\\
 \emph{Case 1.} If $G(x)$ is irreducible  over  $\mathbb{Q}({\sqrt{q}})[x]$ then $\phi(m)=4.2=8$, which implies $m\in\{ 15, ~16, ~20,~ 24,~ 30\}$.
Collecting coefficients of $x$ we get
 \begin{displaymath}
 x^{8}+ \left( \frac {2 a_{2}}{q}-\frac{a_{1}^{2}}{q}\right) x^{6}+ \left( 2+\frac {a_{2}^{2}}{q^2}-\frac {2a_{1}^{2}}{q} \right) x^{4}+ \left( \frac {2a_2}{q}-\frac{a_{1}^{2}}{q}\right) x^{2}+1.
\end{displaymath}
Let\\
 $E_1:= \frac {2a_2}{q}-\frac{a_{1}^{2}}{q},$\\
 $E_2:= 2+\frac {a_{2}^{2}}{q^2}-\frac {2a_{1}^{2}}{q}. $\\

\begin{enumerate}
\item $x^{15}-1=\left( x-1 \right)  \left( 1+x^{4}+x^{3}+x^{2}+x \right)\left( 1+x^{2}+x \right)  \left( 1-x+x^{3}-x^{4}+x^{5}-x^{7}+x^{8} \right)$
\item $x^{16}-1=\left( x-1 \right)  \left( 1+x \right)  \left( 1+{x}^{2} \right)\left( 1+{x}^{4} \right)  \left(1+{x}^{8} \right)$
\item $x^{20}-1=  \left( x-1 \right)  \left(1+x^{4}+x^{3}+x^2+x \right)
  \left( 1+x \right)  \left( 1-x+x^{2}-x^{3}+x^{4} \right)\\ \left( 1+x^{2} \right)  \left( x^{8}-x^{6}+x^{4}-x^{2}+1\right)$
\item $x^{24}-1= \left( x-1 \right)  \left( 1+{x}^{2}+x \right)  \left( 1+x \right)
   \left( 1-x+{x}^{2} \right)  \left( 1+{x}^{2} \right)  \left( {x}^{4}-
   {x}^{2}+1 \right)  \left( 1+{x}^{4} \right)  \left( {x}^{8}-{x}^{4}+1\right)$
\item $x^{30}-1=  \left( x-1 \right)  \left( 1+{x}^{4}+{x}^{3}+{x}^{2}+x \right) \left( 1+{x}^{2}+x \right)  \left( 1-x+{x}^{3}-{x}^{4}+{x}^{5}-{x}^{7}+{x}^{8} \right)  \left( 1+x \right) \\
     \left( 1-x+{x}^{2}-{x}^{3}+{x}^{4} \right)  \left( 1-x+{x}^{2} \right)  \left(1+x-{x}^{3}-{x}^{4}-{x}^{5}+{x}^{7}+{x}^{8} \right) $
\end{enumerate}
So possibilities for $H(x)$ are $x^8+1$ or $ x^{8}-x^{6}+x^{4}-x^{2}+1$ or $ {x}^{8}-{x}^{4}+1$.
\begin{enumerate}
\item If  $H(x)= x^8+1$ then $E_1=E_2=0$. Maple gives $a_{2}$ satisfies $(z^2-4z+2)q$. Since $z^2-4z+2$ is irreducible by Eisentein's 's Criterion, $a_2$ has no integers solution.
\item If $H(x)=  x^{8}-x^{6}+x^{4}-x^{2}+1$ then $E_1=-1,~ E_2=1$. Maple gives $a_2$ satisfies $\frac{-q}{2} +\frac{1}{2}(z^2-10z+5)q$ which again has no integer roots by Eisenstein's criteria.
\item If $H(x)= {x}^{8}-{x}^{4}+1$ then $E_1=0,~ E_2=-1$ which implies $a_1$ satisfies $z^2-6q$ (which has no integer solution for any $q$) or $a_1 \pm \sqrt{2q},~ a_2=q$ which has integer solution for $p=2$.
    Therefore $P(X)=X^4\pm \sqrt{pq}X^3+qX^2\pm q\sqrt{pq}X+q^2$ is the possibility where $p=2$.
\end{enumerate}
\emph{Case 2.}  If $G(x)$ is reducible. Then  by theorem \ref{Red}, $G(x)=F_1(x).F_2(x)$.
The roots of $F_i$  are mth root of unity  where $\phi(m)=2g=4$  if $F_i \in \mathbb{Q}({\sqrt{q}})[x]\setminus \mathbb{Q}[x]$ which case $m \in \{5,8,10,12\}$ else $\phi(m)=g=2$ in which case  $m \in \{3,4,6\}$.
\begin{enumerate}
\item $x^5-1=(x-1)(1+x^4+x^3+x^2+x)$
\item $x^8-1=(x-1)(1+x)(1+x^2)(1+x^4)$
\item $x^{10}-1=(x-1)(1+x^4+x^3+x^2+x)(1+x)(1-x+x^2-x^3+x^4)$
\item $x^{12}-1=(x-1)(1+x^2+x)(1+x)(1-x+x^2)(1+x^2)(x^4-x^2+1)$
\item $x^3 -1=(x-1)(1+x^2+x)$
\item $x^4 -1=(x-1)(1+x)(1+x^2)$
\item $x^6 -1=(x-1)(1+x^2+x)(1+x)(1-x+x^2)$.
\end{enumerate}
This gives following cases for $H(x)$ as in theorem \ref{Red}.
\begin{enumerate}
\item  If $H(x)= 1+2x^4+2x^2+2x^6+x^8$  then $ E_1 = 2,~ E_2 =2$  then $a_1=\pm \sqrt{2q},~ a_2= 2q$ which gives
$P(X)=X^4\pm \sqrt{2q}X^3+2qX^2\pm q\sqrt{2q}X+q^2= (X^2+q)(X^2 \pm\sqrt{2q}X+q)$. But $P(X)$ was assumed to irreducible so this not a possibility.
\item  If $H(x)= x^2+1+x^6+x^8 $   then $  E_1=1, ~ E_2 =0 $ implies $a_1=\pm \sqrt{3q},~ a_2= 2q$ which gives $P(X)=X^4\pm \sqrt{3q}X^3+2qX^2\pm q\sqrt{3q}X+q^2= (X^2+q)(X^2 \pm\sqrt{3q}X+q)$. But $P(X)$  was assumed to irreducible so this not a possibility.
\item  If $H(x)=x^8-x^6 +2x^4-x^2+1$  then $ E_1 = -1,~ E_2 =2$ which implies $a_2$ satisfies $\frac{-q}{2} +\frac{1}{2}(z^2-10z+5)q$ which has no integer roots by Eisenstein's criteria.
\item If $H(x)= x^8-2x^6+3x^4-2x^2+1$  then  $  E_1 = -2, ~ E_2 =3$ then
  \begin{enumerate}
  \item $a_1=0, ~a_2=-q$ then $P(X)=x^4-qx^2+q^2$ which is irreducible if $p\neq 3$, hence is a possibility. If $p =3$ then $P(X)=X^4-qX^2+q^2=(X^2-\sqrt{3q}X+q)(X^2\sqrt{3q}X+q)$.
   \item  $a_1= \pm 2\sqrt{3q},~a_2=5q$ which has integer solutions if $q$ is odd power of 3, but then $P(X)=(X^2\pm \sqrt{3q}X+q)^2$ which is a contradiction to assumption that $P(X)$ was irreducible, hence  not possible.
   \end{enumerate}
\item If $H(x)= 1+x^4+x^2+x^8+x^6$  then  $ E_1 =1,~ E_2 =1$ then $a_1= \pm \sqrt{5q},~ a_2=3q$  or $a_1$ satisfies $z^2-q$ which has  no integer solutions as $q$ is an odd power of prime. So possibility is
 $P(X)=X^4\pm \sqrt{pq}X^3+3qX^2\pm q\sqrt{pq}X+q^2$ where $p=5$.
 \item If $H(x)=1+2x^4+x^8$ then  $E_1 = 0,| E_2 =2$ then
 \begin{enumerate}
 \item $a_1=0,~ a_2=0$ then  $P(X)=X^4+q^2$ which is irreducible if $p\neq 2$ . As $P(X+1)=(X+1)^4+q^2= X^4+4X^3+6X^2+4X+1+q^2$ and $p\neq 2$ implies $1+q^2$ is $2 \mod 4$, hence by Eisentein's criteria is irreducible and is a  possibility for $P(X)$. If $p=2$ then $P(X)=X^4+q^2=(X^2-\sqrt{2q}X+q)(X^2+\sqrt{2q}X+q)$ which is reducible hence not possible.

   \item$a_1= \pm 2\sqrt{2q},~ a_2=4q$, which has integer solutions if $q$ is odd power of 2, but   then $P(X)=(X^2\pm \sqrt{2q}X+q)^2$ which is a contradiction to assumption that $P(X)$ was irreducible, hence  not possible.
\end{enumerate}
\end{enumerate}
\subsection{Case $a_1=0$} 
If $a_1=0$ then $G(x)=x^4+\frac{a_2}{2}x^2+1$. We have 
\begin{enumerate}
\item If $G(x)$ is irreducible, the it is irreducible cyclotomic factor of $x^m-1$  of degree 4 where $\phi(m)=4$. But since in this case $G(x)$  has only even degree terms we have $G(x)=x^4-x^2+1$ in which case $a_2= -q$ which is already dealt above. 
\item If $G(x)$ is reducible  then $G(x)=(1+x+x^2)(1-x+x^2)=1+x^2+x^4$ which gives $a_2=q$ which is also gives a possibility for $P(X)$.\\
\end{enumerate}

Let $P(X)$ be reducible. Then  $P(X)=h(X)^e$ where $h$ is irreducible over $\mathbb{Z}$ with $e|g$ and $h_0= \pm q$ ie;  $P(X)=(X^2+aX \pm q)^2$.
\begin{enumerate}
\item If $a=0$, then $P(X)=(X^2-q)^2$ or $(X^2+q)^2$. But later is not a possible as the corresponding abelian variety is not simple by Tate's theorem, since $x^2+q$ corresponds to dimension $1$ abelian variety. 
\item  If $a \neq 0$ then following theorem \ref{S2} we  have $G(t)=(t^2 \pm 1)+\frac{a}{\sqrt{q}}t$. Then
\begin{enumerate}
\item If constant term is 1, then from discussion on dimension 1 we get $a=\pm \sqrt{2q},~ \pm \sqrt{3q}$  in which case $H(x)$ corresponds to an abelian variety of dimension $1$ see \ref{D1}.
\item If constant term is -1, then from discussion of dimension 1, $a$ has no solution.  
\end{enumerate}

\end{enumerate}
We can conclude above the discussion as a following  theorem.
\begin{theorem}\label{D2}  The characteristic polynomial of  a simple supersingular abelian variety of dimension $2$ over $\mathbb{F}_q$ ($q=p^n$, $n$ odd) is one the following
\begin{enumerate}
\item $p\neq 3: X^4-qX^2+q^2$
\item $X^4 + qX^2+q^2$
\item $p=2:X^4\pm \sqrt{pq}X^3+qX^2\pm q \sqrt{pq}X+q^2$
\item $p=5:X^4\pm \sqrt{pq}X^3+3qX^2\pm q \sqrt{pq}X+q^2$
\item $(X^2-q)^2$
\item $p \neq 2:X^4+q^2$

\end{enumerate}
\end{theorem}
In fact all of them appear, see \cite{Maiser.Nart}.
\section{Dimension 3}
The characteristic polynomial of Frobenius of a abelian variety of  dimension $3$ is given by\\
\begin{displaymath}
P(X)=X^6+a_{1}X^5+a_{2}X^4+a_{3}X^3+a_{2}qX^2+a_{1}q^2X+q^3.
\end{displaymath}
If $P(X)$ is irreducible over $\mathbb{Q}(X)$  then we have following cases.\\
\subsection{ Case $a_{2i+1} \neq 0$}
\emph{Case 1} If  $G(x)$ is irreducible then \\
\begin{displaymath}
H(x)=G(x)G(x)^{\sigma}=(x^6+\frac{a_{2}}{q}x^4+\frac{a_{2}}{q}x^2+1)^2-\frac{1}{q}(a_{1}x^5+\frac{a_{3}}{q}x^3+a_1x)^2
\end{displaymath}
$=x^{12}+(\frac{2a_{2}}{q}-\frac{2a_{1}^2}{q})x^{10}+(\frac{2a_{2}}{q}+\frac{a_{2}^2}{q^2}-\frac{2a_{1}a_{3}}{q^2})x^8+(\frac{-a_{3}^{2}}{q^3}-\frac{2a_{1}^{2}}{q}+\frac{2a_{2}^{2}}{q^2}+2)x^6
+(\frac{2a_{2}}{q}+\frac{a_{2}^2}{q^2}-\frac{2a_{1}a_{3}}{q^2})x^4+(\frac{2a_{2}}{q}-\frac{2a_{1}^2}{q})x^{2}+1$ whose roots are  $m^{th}$ root of unity where $\phi(m)=4g=12$ which implies $m\in \{13,~21,~26,~28,~36,~42\}$. We have following factorization for them.
\begin{enumerate}
\item $x^{13}-1=(x-1)(x^{12} + x^{11} + x^{10} + x^9 + x^8 + x^7 + x^6 + x^5 + x^4 + x^3 + x^2 + x + 1)$
\item $x^{21}-1= (x - 1)(x^2 + x + 1)(x^6 + x^5 + x^4 + x^3 + x^2 + x + 1, 1)(x^{12} - x^{11} + x^9 - x^8 + x^6 - x^4 + x^3 - x + 1)$
\item $x^{26}-1=(x-1)(1+x^{12}+x^{11}+x^{10}+x^9+x^8+x^7+x^6+x^5+x^4+x^3+x^2+x)(1+x)(1-x+x^2-x^3+x^4-x^5+x^6-x^7+x^8-x^9+x^{10}-x^{11}+x^{12})$
\item $x^{28}-1=(x-1)(1+x^6+x^5+x^4+x^3+x^2+x)(1+x)(1-x+x^2-x^3+x^4-x^5+x^6)(1+x^2)(x^{12}-x^{10}+x^8-x^6+x^4-x^2+1)$
\item $x^{36}-1=(x-1)(1+x^2+x)(1+x^6+x^3)(1+x)(1-x+x^2)(1-x^3+x^6)(1+x^2)(x^4-x^2+1)(x^{12}-x^6+1) $
\item $x^{42}-1=(x-1)(1+x^6+x^5+x^4+x^3+x^2+x)(1+x^2+x)(1-x+x^3-x^4+x^6-x^8+x^9-x^{11}+x^{12})(1+x)(1-x+x^2-x^3+x^4-x^5+x^6)(1-x+x^2)(1+x-x^3-x^4+x^6-x^8-x^9+x^{11}+x^{12})$
\end{enumerate}\

 Let $E_1=\frac{2a_{2}}{q}-\frac{2a_{1}^2}{q}$\\
     $E_2=\frac{2a_{2}}{q}+\frac{a_{2}^2}{q^2}-\frac{2a_{1}a_{3}}{q^2}$\\
     $E_3=\frac{-a_{3}^{2}}{q^3}-\frac{2a_{1}^{2}}{q}+\frac{2a_{2}^{2}}{q^2}+2$ \\

Comparing with $H(x)$ with  cyclotomic factor of degree $12$,  we have following possibilities for it.
\begin{enumerate}
\item $H(x)=x^{12}-x^{10}+x^8-x^6+x^4-x^2+1$ in which case $E_1= -1,~ E_2=1,~ E_3=-1$  which gives $a_1= \pm \sqrt{7q}, ~a2=3q,~ a_3=q \pm\sqrt{7q}$ which has integer solutions if and only if $q$ is  odd power of $7$. Other solution for $a_1$ is root of $z^6-21qz^4+35z^2q^2-7q^3$ which has no integers roots for $p\neq 7$ by Eisenstein's criteria.
    If $q=7^n$ then $f(z)=(z^3-7^{\frac{n+1}{2}}z^2-7^{n+1}z-7^{\frac{3n+1}{2}})(z^3+7^{\frac{n+1}{2}}z^2-7^{n+1}z+7^{\frac{3n+1}{2}})$ and each of which have no integer roots.
\item $H(x)=x^{12}-x^6+1$ we have $E_1= 0,~ E_2=0,~ E_3=-1$ which gives one solution as    $a_1= 0, ~a_2=0,~ a_3=q \sqrt{3q}$ which has integer solutions if and only if $q$ is  odd power of $3$. The other solution for $a_1$ is root of $z^6-6qz^4+9z^2q^2-3q^3$ which has no integers roots for $p\neq 3$ by Eisenstein's criteria. If  $p =3$ then if $q=3^n$  then $z^6-6qz^4+9z^2q^2-3q^3=(z^3-3^{n+1}z-3^{\frac{3n+1}{2}})(z^3-3^{n+1}z+3^{\frac{3n+1}{2}})$, where it is easy to check none of this factors have integer solutions.
\end{enumerate}
\emph{Case 2.} $G(x)$ is reducible. Then  by \ref{Red} $G(x)=F_1(x).F_2(x)$.
The roots of $F_i$  are mth root of unity  where $\phi(m)=2g=6$  if $F_i \in \mathbb{Q}({\sqrt{q}})[x]\setminus \mathbb{Q}[x]$ else $\phi(m)=g=3$. Since  $\phi(m)=g=3$ has no solution for $m$ both $ F_i \in \mathbb{Q}({\sqrt{q}})[X] \setminus \mathbb{Q}[x]$. If $\phi(m)=6$ then $m \in \{7, 9, 14, 18\}$, which has following expansions.
\begin{enumerate}
\item $x^7-1=  \left( x-1 \right)  \left( {x}^{6}+{x}^{5}+{x}^{4}+{x}^{3}+{x}^{2}+x+1 \right)$.
\item $x^9-1=  \left( x-1 \right)  \left( {x}^{2}+x+1 \right)  \left( {x}^{6}+{x}^{3}+1 \right)$.
\item $x^{14}-1=\left( x-1 \right)\left( {x}^{6}+{x}^{5}+{x}^{4}+{x}^{3}+{x}^{2}+x+1 \right)    \left( x+1 \right)  \left(1-x+{x}^{2}-{x}^{3}+{x}^{4}-{x}^{5}+{x}^{6} \right)$.
\item  $x^{18}-1= \left( x-1 \right)  \left( {x}^{2}+x+1 \right)  \left( {x}^{6}+{x}^{3}+1 \right)  \left( x+1 \right)  \left( {x}^{2}-x+1 \right)  \left( {x}^{6}-{x}^{3}+1 \right)$.
\end{enumerate}
Comparing with  $H(x)=G(x)G(x)^{\sigma}=F_1(x)F_1(x)^{\sigma}F_2(x)F_2(x)^{\sigma}$ we get following possibilities.
\begin{enumerate}
\item $H(x)=x^{12}+x^{10}+x^8+x^6+x^4+x^2+1$ then $E_1= 1,~ E_2=1,~ E_3=1$. One of the solution of  $a_1$ satisfies $z^2-q$ which has no integer solutions as $q$ is odd power of prime. The other solution is the root $f(z)=z^6-19qz^4+83z^2q^2-q^3$. We claim that it has no integer solutions. Suppose it has integer solution then it should be solution modulo 3.
    $q=1$ then  $f(z)\mod 3=(z^3+2z^2+1)(z^3-z^2+2)$~, $q=2$ then $f(z)\mod 3=z^6+z^4+2z^2+1$,$q=1$ then $f(z)\mod 5 =(z^3+4z^2+z+1)(z^3+z^2+z+4)$, $q=-1$ then $f(z)\mod 5 =z^6+4z^4+3z^2+1 $ has no solutions, hence by lemma ($\mod 3 \mod 5$ test) it has no integer solutions.
\item If $H(x)=x^{12}+x^6+1$ then $E_1= 0,~ E_2=0, ~E_3=1$. One of the solution of  $a_3$ satisfies $z^2-q$ which has no integer solutions as $q$ is odd power of prime. The other solution $a_1$ is the twice the root $f(z)=z^6-6qz^4-9z^2q^2-q^3$. If
    $q=1$ then  $f(z)\mod 3=(z^3+1)(z^3+2),$
    $q=2$ then $f(z)\mod 3=z^6+1$,
    $q=1$ then $f(z)\mod 5 =z^6+z^4+4z^2+1 $,
    $q=-1$ then $f(z)\mod 5 =z^6+z^4+4z^2+1 $
    has no solutions, hence by lemma ($\mod 3 \mod 5$ test) it has no integer solutions.
\end{enumerate}
\subsection{ Case $a_{2i+1} =0$}
If  $a_{2i+1} = 0$ then \\
\emph{Case 1} $G(x)=x^6+\frac{a_{2}}{q}x^4+\frac{a_{2}}{q}x^2+1$ is irreducible over $\mathbb{Q}$ then it has roots as mth root of unity where $\phi(m)=6$ hence $m \in \{7,~9,~14,~18\}$. The possibility of such $G(x)$ is already mention in case 2. and none of them have this form.\\
\emph{Case 2} If $G(x)$ is reducible then $G=F_1F_2$   such that $F_i$ is irreducible factor of $x^m-1$  such that $\phi(m)=3$, which has no solution for $m$. \\

\subsection{$ P(X)$ is reducible}
If $P(X)$ is reducible then $P(X)=h(X)^e$, where $e|3$ implies $e=3$. Therefore $h(X)=(X^2+aX+q)$ and it is not possible for  $P(X)$ to correspond to a simple abelian variety as already discussed in dimension 1 case.
Hence we have,
\begin{theorem}\label{D3} The characteristic polynomial of a simple supersingular abelian variety of dimension $3$ over $\mathbb{F}_q$ ($q=p^n$, $n$ odd) is one the following
\begin{enumerate}
\item $p=3: X^6 \pm q\sqrt{pq}X^3+q^3$
\item $p=7:X^6\pm\sqrt{pq}X^5+3qX^4\pm q\sqrt{pq}X^3+3q^2X^2\pm q^2\sqrt{pq}X+q^3$
\end{enumerate}
\end{theorem}
In fact of them occur, see \cite{Nart}.
\section{Dimension 4}

The characteristic polynomial of Frobenius of an abelian variety of dimension $4$ is given by
\begin{displaymath}
P(X)=X^8+a_{1}X^7+a_{2}X^6+a_{3}X^5+a_{4}X^4+a_{3}qX^3+a_{2}q^2X^2+a_{1}q^3X+q^4.
\end{displaymath}
Let  $P(X)$ is irreducible then we have following cases.
\subsection{Case $a_{2i+1} \neq 0$} 
If  $a_{2i+1} \neq 0$ then we have following cases,\\
\emph{Case 1.} If $G(x)$ is irreducible as in theorem \ref{S2}, we have

\begin{displaymath}
H(x)=\left( x^{8}+{\frac {{ a_2}x^{6}}{q}}+{\frac {{ a_4}x^{4}}{{q}^{2}}}+{\frac {{ a_2}x^{2}}{q}}+1 \right) ^{2}- \frac{1}{q}\left( {a_1}x^{7}+{\frac {{ a_3}x^{5}}{q}}+{\frac {{ a_3}x^{3}}{q}}+{ a_1}x \right) ^{2}
\end{displaymath}

\begin{eqnarray*}\nonumber
=x^{16}+ \left( {\frac {{{ a_1}}^{2}}{q}}+{\frac {{ 2a_2}}{q}} \right) x^{14}+ \left( 2{\frac {{ a_4}}{{q}^{2}}}+{\frac {{{
 a_2}}^{2}}{{q}^{2}}}-2{\frac {{ a_1}{ a_3}}{{q}^{2}}}\right) x^{12}+ \left( 2{\frac {{ a_2}}{q}}-{\frac {{{ a_3}}^
{2}}{{q}^{3}}}+2{\frac {{ a_4}{ a_2}}{{q}^{3}}}-2{\frac {{a_1}{ a_3}}{{q}^{2}}} \right) x^{10}+\\
 \left( {\frac {{{ a_4}}^{2}}{{q}^{4}}}-2{\frac {{{ a_1}}^{2}}{q}}+2+2{\frac {{{ a_2}}^{2}}{{q}^{2}}}-2{\frac {{{ a_3}}^{2}}{{q}^{3}}} \right) x^{8}
+ \left( 2{\frac {{ a_2}}{q}}-{\frac {{{ a_3}}^{2}}{{q}^{3}}}+2{\frac {{ a_4}{ a_2}}{{q}^{3}}}-2{\frac {{ a_1}{ a_3}}
{{q}^{2}}} \right) x^{6}+ \left( 2{\frac {{ a_4}}{{q}^{2}}}+{\frac {{{ a_2}}^{2}}{{q}^{2}}}-2{\frac {{ a_1}{ a_3}}{{q}^{2
}}} \right) x^{4}+\\
 \left( -{\frac {{{ a_1}}^{2}}{q}}+2{\frac {{a_2}}{q}} \right) x^{2}+1
\end{eqnarray*} 
whose roots are mth roots of unity where $\phi=4g= 16$ which implies $m \in \{17, 32, 34, 40, 48, 60\}$. Each of which has following factorization.
\begin{enumerate}
\item $x^{17}-1 = (x-1)(x^{16} + x^{15} + x^{14}+ x^{13} + x^{12} + x^{11} + x^{10} + x^9 + x^8 + x^7 + x^6 + x^5 + x^4 + x^3 + x^2 + x + 1)$
\item $x^{32} -1 =(x-1)(1+x)(1+x^2)(1+x^4)(1+x^8)(1+x^{16})$

\item $x^{34}-1=(x-1)(1+x^{16}+x^{15}+x^{14}+x^{13}+x^{12}+x^{11}+x^{10}+x^9+x^8+x^7+x^6+x^5+x^4+x^3+x^2+x)(1+x)(1-x+x^2-x^3+x^4-x^5+x^6-x^7+x^8-x^9+x^{10}-x^{11}+x^{12}-x^{13}+x^{14}-x^{15}+x^{16})$
\item $x^{40}-1 =(x-1)(1+x^4+x^3+x^2+x)(1+x)(1-x+x^2-x^3+x^4)(1+x^2)(x^8-x^6+x^4-x^2+1)(1+x^4)(x^{16}-x^{12}+x^8-x^4+1)$
\item $x^{48}-1=(x-1)(1+x^2+x)(1+x)(1-x+x^2)(1+x^2)(x^4-x^2+1)(1+x^4)(x^8-x^4+1)(1+x^8)(x^{16}-x^8+1)$
\item $x^{60}-1=(x-1)(1+x^4+x^3+x^2+x)(1+x^2+x)(1-x+x^3-x^4+x^5-x^7+x^8)(1+x)(1-x+x^2-x^3+x^4)(1-x+x^2)(1+x-x^3-x^4-x^5+x^7+x^8)(1+x^2)(x^8-x^6+x^4-x^2+1)(x^4-x^2+1)(x^{16}+x^{14}-x^{10}-x^8-x^6+x^2+1)$
\end{enumerate}
Let
$E_1= \frac {{{ a_1}}^{2}}{q}+\frac {{ 2a_2}}{q}$\\
 $E_2=\frac {{2 a_4}}{{q}^{2}}+\frac {{{a_2}}^{2}}{{q}^{2}}-\frac {{2 a_1}{ a_3}}{{q}^{2}}$\\
 $E_3= \frac {{2 a_2}}{q}-{\frac {{{ a_3}}^{2}}{{q}^{3}}}+\frac {{ 2a_4}{ a_2}}{{q}^{3}}-\frac {{2a_1}{ a_3}}{{q}^{2}} $\\
 $E_4=\frac {{{a_4}}^{2}}{{q}^{4}}-\frac {{ 2a_1}^{2}}{q}+2+\frac {{{2a_2}}^{2}}{{q}^{2}}-\frac {{{2a_3}}^{2}}{{q}^{3}} $

Comparing $H(x)$ with degree $16$ irreducible cyclotomic factor we have following possibilities for $H(x)$. If,

 \begin{enumerate}
 \item $H(x)=1+x^{16}$ then $E_1=E_2=E_3=E_4=0$ then $a_1$ satisfies $z^8-32z^6q+160z^4q^2-256z^2q^3+128q^4$ or
$z^8-32z^6q+288z^4q^2-768z^2q^3+128q^4$ which has no integer solutions by $\mod 3 \mod 5$ test.
\item $H(x)=x^{16}-x^{12}+x^8-x^4+1$ then $E_1=E_3=0$, $E_2=-1$ and $E_4=1$. Maple gives following solutions for $a_i$.
     \begin{enumerate}
     \item $a_1=\pm\sqrt{2q},~a_2=q,~a_3=0,~a_4=-q^2$ is one of the possible solution if $q$ is odd power of $2$.
     \item $a_1$ is root of $z^2-10q$ which has no integer solution for any $q$.
     \item $a_1$ is root of  $z^4-20z^2q+20q^2$  has no integer solution for any $q$ (by mod 3 mod 5 test).
     \end{enumerate}
\item $H(x)=x^{16}-x^8+1$ then $E_1=E_2=E_3=0$ and $E_4=-1$. This gives $a_1$ is root of $8q^2-8Z^2q+Z^4$ or $72q^2-24Z^2q+Z^4 $ or $Z^8-32Z^6q+176Z^4q^2-256Z^2q^3+64q^4$. None of these equations have integer solutions (by mod 3 mod 5 test).
\item $H(x)= x^{16}+x^{14}-x^{10}-x^8-x^6+x^2+1$ then $E_1=1,~E_2=0,~E_3=-1, ~E_4=-1$ then we have following
      \begin{enumerate}
      \item $a_1= \pm \sqrt{3q}, ~a_2=2q, ~a_3= q a_1,~a_4=q^2$ which has integer solutions only if q is odd power of $3$.
       \item $a_1$ satisfies $Z^2-15q$ or $5q^2-10Z^2q+Z^4$ or $Z^8-28Z^6q+134Z^4q^2-92Z^2q^3+q^4$ which has no integer solutions by $\mod 3 \mod 5$ test.
    \end{enumerate}

 \end{enumerate}

\emph{Case 2.} If $G(x)$ is reducible. Then  by  theorem \ref{Red}, $G(x)=F_1(x).F_2(x)$.
 If $F_i \in \mathbb{Q}({\sqrt{q}})[x] \setminus \mathbb{Q}[x]$ then its roots are  mth root of unity  where $\phi(m)=2g=8$  where $m \in \{15, 16, 20, 24, 30\}$ in which case $F_i(x)F_i(x)^{\sigma}$ is an irreducible cyclotomic factor of $x^m-1$  of degree $8$.

If $F_i \in \mathbb{Q}[x]$ then its roots are  mth root of unity  where  $\phi(m)=g=4$ in which case $m \in \{5,8,10,12\}$ and  $F_i$ is an irreducible cyclotomic factor of $x^m-1$ of degree $4$.
 Now $H(x)=F_1 F_{1}^{\sigma}F_2 F_{2}^{\sigma}$ where not both $F_i \in \mathbb{Q}[x]$ as discussed earlier. Since $H(x)$ has only even degree terms, we look at all possibilities $F_1 F_{1}^{\sigma}F_2 F_{2}^{\sigma}$ from above, all of them are listed  below. If,
 \begin{enumerate}
  \item $H(x)=1+2x^8+x^{16}$ then $E_1=E_2=E_3=0$ and $E_4=2$. Solving we get $a_1=a_2=a_3=a_4=0$ which is one of the possibility. The other possibilities are $a_1$ is root of $Z^4-4Z^2q+2q^2$ or $Z^4-8Z^2q+8q^2$ which have no integer solutions by mod 3 mod 5 test.
  \item $H(x) =x^{16}-x^{14}+x^{12}-x^{10}+2x^8-x^6+x^4-x^2+1$ then $E_1=E_3=-1$, $E_2=1$ and $E_4=2$ which gives $a_1$ satisfies $Z^{16}-72Z^{14}q+1836Z^{12}q^2 -21336Z^{10}q^3+120854q^4Z^8-334008q^5Z^6+393804q^6Z^4-107304Z^2q^7+7921q^8$ which has no integer solutions by $\mod 3 \mod 5$ test.
  \item $H(x)= x^{16}-x^{12}+2x^8-x^4+1$ then $E_1= E_3 = 0$, $E_2= -1$ and $E_4=2$ which gives $a_1$ satisfies
  $Z^8-40Z^6q+392Z^4q^2-608Z^2q^3+16q^4$ or $Z^8-24Z^6 q+136Z^4q^2-160Z^2q^3+16q^4$. None of these have integer solutions by $\mod 3, \mod 5$ test.
  \item $H(x) =x^{16}-2x^{14}+3x^{12}-4x^{10}+5x^8-4x^6+3x^4-2x^2+1$. We have $E_1=-2, ~E_2=3, ~E_3=-4$ and $E_4=5$ which gives $a_1=a_3=0$ and $a_2=-q, ~a_4=q^2$ which is a possibility. The other solutions $a_1$ are roots of $2(Z^4-10Z^2q+5q^2)$ or $2(Z^4-5Z^2q+5q^2)$ none of which is an integer by mod 3 mod 5 test.
  \item $H(x) =x^{16}-x^{14}+x^8-x^2+1$ then $E_1=-1, ~E_2=E_3=0$ and $E_4=1$ in which case $a_1$ is root of $Z^8-28Z^6q+174Z^4q^2-332Z^2q^3+121q^4$ or $Z^8-44Z^6q+446Z^4q^2-1084Z^2q^3+361q^4$ which has no integer solutions by mod 3 mod 5 test.
  \item $H(x) =x^{16}-2x^{12}+3x^8-2x^4+1$ then $E_1= E_3=0$, $E_2=-2$ and $E_4=3$. The solutions for $a_i$'s are
          \begin{enumerate}
           \item $a_1=a_2=a_3=0$ and $a_4=-q^2$ which gives $P(X)=X^8-q^2X^4+q^4$ which is irreducible when $p \neq 2$. When $p=2$ then$P(X)=X^8-q^2X^4+q^4=(X^4-\sqrt{2q}X^3+qX^2-q\sqrt{2q}X+q^2)(X^4+\sqrt{2q}X^3+qX^2+q\sqrt{2q}X+q^2)$ which means $A$ is  not simple by Tate's theorem.
           \item $a_1=\pm2\sqrt{2q},~a_2=4q, ~a_3=\pm 4q\sqrt{2q}, ~a_4=7q^2$ which integer solutions only if $q$ is odd power of $2$ . Then $P(X)=X^8\pm 2 \sqrt{2q}X^7+4qX^6 \pm 4 q\sqrt{2q}X^5+7q^2X^4 \pm 4q^2 \sqrt{2q}X^3+  4q^3X^2 \pm 2q^2\sqrt{2q}X+q^4=(X^4 \pm \sqrt{2q}X^3+qX^2\pm q\sqrt{2q}X+q^2)^2$  which implies $A$ is not simple by Tate's theorem.
           \item The other solutions for $a_1$ are roots of  $Z^2-6q$  or $2Z^4-4Z^2q+q^2$  for  which there is no integer solutions.
          \end{enumerate}
  \item $H(x) =x^{16}-x^{14}+x^{10}-x^8+x^6-x^2+1$ then $E_1=E_4=-1$, $E_2=0$ and $E_3=1$. One of the solutions are $a_1= \pm \sqrt{5q}, ~a_2=2q,~ a_3=qa_1, a_4=3q^2$. The other solutions for $a_1$ are roots of $Z^2-q$, $45q^2-30Z^2q+Z^4$ which has no integer roots for q, an odd power of prime p.

\item $H(x)=1+2x^4+2x^8+2x^{12}+x^{16}$ then $E_1=0, ~E_2=E_3= E_4=2$ then $a_1$ satisfies $-32Z^6q+64q^4+112Z^4q^2-128Z^2q^3+Z^8$ or $-32Z^6q+1600q^4+304Z^4q^2-1152Z^2q^3+Z^8$ which has no integer solutions by mod 3 mod 5 test.
\item $H(x)=x^{16}-2x^{14}+3x^{12}-2x^{10}+2x^8-2x^6+3x^4-2x^2+1$ then $ E_1=E_3=-2,~ E_2=3,~ E_4=2$. This gives $a_1$ satisfies $8q^2-8Z^2q+Z^4$ or $3136q^4-5120Z^2q^3+1136Z^4q^2-64Z^6q+Z^8$ none of which has integer roots.
\item $H(x)=x^{16}+3x^{12}-x^{14}-3x^{10}+4x^8-3x^6+3x^4-x^2+1$ then $E_1=-1,~ E_2=3,~E_3=-3, E_4=4$. Then $a_1$ is root of $5q^2-10Z^2q+Z^4$ or $121q^4-1988Z^2q^3+654Z^4q^2-52Z^6q+Z^8$ which has no integer solutions.
\item $H(x)=x^{16}-3x^{14}+6x^{12}-8x^{10}+9x^8-8x^6+6x^4-3x^2+1$ then $E_1=-3,~ E_2=6,~E_3=-8, ~E_4=9$. In this case $a_1$ satisfies $5q^2-10Z^2q+Z^4$ or $841q^4-5812Z^2q^3+1214Z^4q^2-68Z^6q+Z^8$ which has no integer solutions.
\item $H(x)=x^{16}+x^{12}+x^4+1$ then $E_1=E_3= E_4=0$ and $E_2=1$. We have following solutions
\begin{enumerate}
\item $a_1=\pm\sqrt{2q}, ~a_2=q,~ a_3=\pm q\sqrt{2q},~a_4=\pm 2q^2$ which gives
 $P(X)=X^8\pm \sqrt{2q}X^7+qX^6 \pm q\sqrt{2q}X^5 +2 q^2X^4\pm q^2\sqrt{2q}X^3+q^3X^2 \pm q^3\sqrt{2q}X+q^4=(X^2+\sqrt{2q}X+q)(X^2-\sqrt{2q}X+q)(X^4\pm \sqrt{2q}X^3+qX^2 \pm q\sqrt{2q}X+q^2)$ which implies corresponding abelian variety is not simple.
\item $a_1=3\sqrt{2q}, ~a_2=9q,~a_3=\pm9q\sqrt{2q},~a_4=14q^2$ in which case $P(X)=X^8\pm 3\sqrt{2q}X^7+9qX^6  \pm 9q\sqrt{2q}X^5+14q^2X^4 \pm 9q^2\sqrt{2q}X^3 + 9q^3X^2 \pm 3q^3\sqrt{2q}X+q^4=(X^4\pm\sqrt{2q}X^3+qX^2\pm q\sqrt{2q}X+q^2)(X^2\pm \sqrt{2q}X+q)^2$ which implies corresponding abelian variety is not simple.
\item The other solutions to $a_1$ are root of $Z^2-6q$ or $Z^2-28Z^2q+4q^2$ which has no integer solution by mod 3 mod 5 test.
\end{enumerate}
\item $H(x)=x^{16}-2x^{14}+2x^{12}-x^8+2x^4-2x^2+1$ then  $E_1=-2,~ E_2=2,~E_3=0,~ E_4=-1$. One of the solutions are
$a_1= \pm \sqrt{2q}, ~a_2=a_3=0,~a_4=q^2$ which has integer solutions only if q is odd power of 2 in which case
$P(X)=X^8 \pm \sqrt{2q}X^7+q^2X^4\pm q^3\sqrt{2q}X+q^4=(X^4\pm\sqrt{2q}X^3+qX^2\pm q\sqrt{2q}X+q^2)(X^4\pm q X+q^2)$ which implies by Tate's theorem that $A$ is not simple. The other solutions of $a_1$ are roots of $Z^2-6q$ (doesn't has integer solution for an odd power of p) or $Z^4-36Z^2q+36q^2$ or $100q^2-28Z^2q+Z^4$ which is  not an integer solutions by $\mod 3 \mod 5$ test.
\end{enumerate}

\subsection{Case  $a_{2i+1}=0$}
If  $a_{2i+1} \neq 0$, then $G(x)=x^{8}+{\frac {{ a_2}x^{6}}{q}}+{\frac {{ a_4}x^{4}}{{q}^{2}}}+{\frac {{ a_2}x^{2}}{q}}+1$, with roots as mth root of unity where,
\begin{enumerate}
\item If $G(x)$ is irreducible then $\phi(m)=8$ which gives $m \in \{15,16,20, 24,30\}$. Since $G(x)$ is irreducible it is degree $8$ irreducible factor with even power terms of $x^m-1$. Comparing we get following
    \begin{enumerate}
    \item If $G(x)=1+x^8$ we get all $a_i=0$ and $P(X)=x^8+q^4$.
    \item If $G(x)=x^8-x^6+x^4-x^2+1$ then $a_2= -q$ and $a_4=q^2$ this is a possibility.
    \item If $G(x)=x^8-x^4+1$ then $a_2= 0$ and $a_4=-q^2$, this case is already discussed earlier.
    \end{enumerate}
    \item If $G(x)$ is reducible then it is product of two degree $4$ irreducible factors of $x^m-1$ above. By looking at products  with only even terms and comparing with $G(x)$ we get
        \begin{enumerate}
        \item $G(x)=x^8-x^6+2x^4-x^2+1$ which gives $a_2=-q, ~a_4=2q^2$. Therefore  $P(X)=X^8-X^6q+2X^4q^2-q^3X^2+q^4=(X^4+q^2)(X^4-qX^2+q^2)$ which is not irreducible.
        \item $G(x)=x^8-2x^6+3x^4-2x^2+1$ which gives $a_2=-2q, ~a_4=3q^2$. Therefore $P(X)=X^8-2X^6q+3X^4q^2-2q^3X^2+q^4=(X^4-qX^2+q^2)^2$ which is not irreducible.
        \item $G(x)=x^8+x^6+x^4+x^2+1$ which gives $a_2=q, ~a_4=q^2$ which gives $P(X)= X^8+ qX^6+q^2X^4+q^3X^2+q^4$ which is irreducible and is a possibility when $p \neq 5$. When $p=5$ then
            $P(X)=(X^4-\sqrt{5q}X^3+3qX^2-q\sqrt{5q}X+q^2)(X^4+\sqrt{5q}X^3+3qX^2+q\sqrt{5q}X+q^2)$ which is not irreducible.
        \item $G(x)=1+2x^4+x^8$   which gives $a_2=0,~a_4=2q^2$ in which case $P(X)=(X^4+q^2)^2$ is not irreducible.
        \end{enumerate}
\end{enumerate}
\subsection{$P(X)$ is reducible}
If $P(X)$ is irreducible then $P(X)=h(X)^e$ where $e|4$ therefore $e=2$ or $4$. 
If $e=4$ then $h(X)=X^2+aX \pm q$ and it is not possible for  $P(X)$ to correspond to a simple abelian variety as already discussed in dimension 1 case.
If $e=2$ then $h(X)=(X^4+bX^3+cX^2+dX \pm q^2)$. On doing the transformation as $\frac{X}{\sqrt{q}}\rightarrow t $ we get $G(t)=(t^4+\frac{c}{q}t^2 \pm 1)-\frac{1}{\sqrt{q}}(bt^3+\frac{d}{q}t)$ or $H(t)=t^8+(\frac{2c}{q}-\frac{b^2}{q})t^6+(\frac{c^2}{q^2}-\frac{2bd}{q^2}+2)t^4+(\frac{2c}{q}- \frac{d^2}{q^3})t^2 \pm1$. We have following cases
\begin{enumerate} 
\item If constant term of $h(X)$ has minus sign, then $b~,c~d$ have no integer solutions.
\item If constant term of $h(X)$ has plus sign, then $h(X)$ corresponds to characteristic polynomial of dimension  $2$ supersingular abelian variety given in the list see \ref{D2}. Since all of them occur, by Tate's theorem the abelian variety corresponding to $P(X)$ is not simple. 
\end{enumerate}
Hence we have the following theorem.
\begin{theorem} The characteristic polynomial of the a simple supersingular abelian variety of dimension $4$ over $\mathbb{F}_q$ ($q=p^n$, $n$ odd) is one of the following
\begin{enumerate}
\item $p=2: X^8\pm \sqrt{2q}X^7+qX^6-q^2X^4+q^3X^2 \pm q^3\sqrt{2q}X+q^4$
\item $p=3: X^8\pm \sqrt{3q}X^7+2qX^6\pm q \sqrt{3q}X^5+q^2X^4\pm q^2 \sqrt{3q}X^3+2q^3X^2\pm q^3\sqrt{3q}X^+q^4$
\item $X^8+q^4$
\item $X^8- qX^6+q^2X^4- q^3X^2+q^4$
\item $ p \neq 5: X^8+ qX^6+q^2X^4+q^3X^2+q^4$
\item $p \neq 2 : X^8-q^2X^4+q^4$
\item $p=5: X^8\pm \sqrt{5q}X^7+2qX^6\pm q \sqrt{5q}X^5+3q^2X^4\pm q^2 \sqrt{5q}X^3+2q^3X^2\pm q^3\sqrt{5q}X^+q^4$
\end{enumerate}
\end{theorem}

\begin{theorem}
 All of the polynomials listed above occur as characteristic polyonomial of Frobenuis of dimension 4.
\end{theorem}

\begin{proof}
\begin{enumerate}
\item If $P(X)=X^8+q^4$, substitute  $y=qX^2$. Then we get $P(X)=q^4(y^4+1)$. The polynomial $y^4+1$ is 8th cyclotomic polynomial.
\begin{enumerate}
\item If $p \equiv 1 \mod 8 $ then $y^4+1=\displaystyle\prod_{i=1}^{4}(y-\alpha_i)$ over  $\mathbb{Q}_p$ with $v_p(\alpha_i)=0$. Since $n$ is odd we get\\
 $P(X)=\displaystyle\prod_{i=1}^{4}(X^2-\alpha_i q)$. Hence we have 4 invariants with $inv_{\mathfrak{p}_i} (End_k (A)\otimes \mathbb{Q}) \equiv 0 \mod \mathbb{Z}$ which shows the $\dim A= 4$.
\item  If $p \equiv 3,5,7  \mod 8$  then $y^4+1=\displaystyle\prod_{i=1}^{2}(y^2+\beta_iy+\alpha_i)$ with $v_p(\alpha_i)=0$. Since $n$ is odd we get\\
 $P(X)=\displaystyle\prod_{i=1}^{2}(X^4+\beta_i X^2+\alpha_i q^2)$. Hence we have 2 invariants with $inv_{\mathfrak{p}_i} (End_k (A)\otimes \mathbb{Q}) \equiv 0 \mod \mathbb{Z}$ which shows the $\dim A= 4$.

\end{enumerate} 
\item If $P(X)=X^8-q^2X^4+q^4$, substitute  $y=qX^2$. Then we get $P(X)=q^4(y^4-y^2+1)$. The polynomial $y^4-y^2+1$ is 12th cyclotomic polynomial.
\begin{enumerate}
\item If $p \equiv 1 \mod 12 $ then $y^4+1=\displaystyle\prod_{i=1}^{4}(y-\alpha_i)$ over  $\mathbb{Q}_p$ with $v_p(\alpha_i)=0$. Since $n$ is odd we get\\
 $P(X)=\displaystyle\prod_{i=1}^{4}(X^2-\alpha_i q)$. Hence we have 4 invariants with $inv_{\mathfrak{p}_i} (End_k (A)\otimes \mathbb{Q}) \equiv 0 \mod \mathbb{Z}$ which shows the $\dim A= 4$.
\item  If $p \equiv 5,7,11  \mod 12$  then $y^4+1=\displaystyle\prod_{i=1}^{2}(y^2+\beta_i+\alpha_i)$ with $v_p(\alpha_i)=0$. Since $n$ is odd we get\\
 $P(X)=\displaystyle\prod_{i=1}^{2}(X^4+\beta_i X^2+\alpha_i q^2)$. Hence we have 2 invariants with $inv_{\mathfrak{p}_i} (End_k (A)\otimes \mathbb{Q}) \equiv 0 \mod \mathbb{Z}$ which shows the $\dim A= 4$.
\end{enumerate} 
\item If $P(X)=X^8- qX^6+q^2X^4- q^3X^2+q^4$, substitute  $y=qX^2$. Then we get $P(X)=q^4(y^4-y^3+y^2-y+1)$. The polynomial $y^4-y^3+y^2-y+1$ is 10th cyclotomic polynomial.
\begin{enumerate}
\item If $p \equiv 1 \mod 10 $ then $y^4-y^3+y^2-y+1=\displaystyle\prod_{i=1}^{4}(y-\alpha_i)$ over  $\mathbb{Q}_p$ with $v_p(\alpha_i)=0$. Since $n$ is odd we get\\
 $P(X)=\displaystyle\prod_{i=1}^{4}(X^2-\alpha_i q)$. Hence we have 4 invariants with $inv_{\mathfrak{p}_i} (End_k (A)\otimes \mathbb{Q}) \equiv 0 \mod \mathbb{Z}$ which shows the $\dim A= 4$.
\item If $p \equiv 3, 7  \mod 10$ then $y^4-y^3+y^2-y+1$ is irreducible, hence there is one invariant with  $inv_{\mathfrak{p}_i} (End_k (A)\otimes \mathbb{Q}) \equiv 0 \mod \mathbb{Z}$ which shows the $\dim A= 4$.
\item  If $p \equiv 9  \mod 10$  then $y^4+1=\displaystyle\prod_{i=1}^{2}(y^2+\beta_i y+\alpha_i)$ with $v_p(\alpha_i)=0$. Since $n$ is odd we get\\
 $P(X)=\displaystyle\prod_{i=1}^{2}(X^4+\beta_i X^2+\alpha_i q^2)$. Hence we have 2 invariants with $inv_{\mathfrak{p}_i} (End_k (A)\otimes \mathbb{Q}) \equiv 0 \mod \mathbb{Z}$ which shows the $\dim A= 4$.
\end{enumerate} 
\item If $P(X)=X^8+ qX^6+q^2X^4+q^3X^2+q^4$, substitute  $y=qX^2$. Then we get $P(X)=q^4(y^4y^3+y^2+y+1)$. The polynomial $y^4+y^3+y^2+y+1$ is 5th cyclotomic polynomial.
\begin{enumerate}
\item If $p \equiv 1 \mod 5 $ then $y^4-y^3+y^2-y+1=\displaystyle\prod_{i=1}^{4}(y-\alpha_i)$ over  $\mathbb{Q}_p$ with $v_p(\alpha_i)=0$. Since $n$ is odd we get
 $P(X)=\displaystyle\prod_{i=1}^{4}(X^2-\alpha_i q)$. Hence we have 4 invariants with $inv_{\mathfrak{p}_i} (End_k (A)\otimes \mathbb{Q}) \equiv 0 \mod \mathbb{Z}$ which shows the $\dim A= 4$.
\item If $p \equiv 2,3  \mod 5$ then $y^4+y^3+y^2+y+1$ is irreducible, hence there is one invariant with  $inv_{\mathfrak{p}_i} (End_k (A)\otimes \mathbb{Q}) \equiv 0 \mod \mathbb{Z}$ which shows the $\dim A= 4$.
\item  If $p \equiv 4  \mod 5 $  then $y^4+1=\displaystyle\prod_{i=1}^{2}(y^2+\beta_iy+\alpha_i)$ with $v_p(\alpha_i)=0$. Since $n$ is odd we get\\
 $P(X)=\displaystyle\prod_{i=1}^{2}(X^4+\beta_i X^2+\alpha_i q^2)$. Hence we have 2 invariants with $inv_{\mathfrak{p}_i} (End_k (A)\otimes \mathbb{Q}) \equiv 0 \mod \mathbb{Z}$ which shows the $\dim A= 4$.
\end{enumerate}
\item If $P(X)= X^8\pm \sqrt{2q}X^7+qX^6-q^2X^4+q^3X^2 \pm q^3\sqrt{2q}X+q^4$, $p =2$ is irreducible in $\frac{\mathbb{Z}}{3\mathbb{Z}}$, hence irreducible over $\mathbb{Q}$.
$P(X/\sqrt{q})= X^8\pm \sqrt{2}X^7+X^6-X^4+X^2 \pm \sqrt{2}X+q^4$ is irreducible over $\mathbb{Q}(\sqrt{2})[X]$ with one of the roots as $\zeta_{40}$ and splitting field $\mathbb{Q}(\zeta_{40})$. 
We have $[ \mathbb{Q}(\sqrt{2}):\mathbb{Q}(\zeta_{40})]=8$. Passing through completion  and using theorem \ref{padics} we get,\\
\begin{displaymath}
[\mathbb{Q}_{2}:\mathbb{Q}_{2}(\sqrt{2})][\mathbb{Q}_{2}(\sqrt{2}):\mathbb{Q}_{2}(\zeta_{40})]=[\mathbb{Q}_{2}:\mathbb{Q}_{2}(\zeta_{40})]=16.
\end{displaymath}
Therefore $[ \mathbb{Q}_2(\sqrt{2}):\mathbb{Q}_2(\zeta_{40})]=8$. But $P(X/\sqrt{q}) \in \mathbb{Q}_{2}(\sqrt{2})$ has degree $8$  and $\zeta_{40}$ as one of the  roots. 
This implies  $P(X/\sqrt{q})$ and hence $P(X)$ are irreducible over $\mathbb{Q}(\sqrt{2})$ which implies $P(X)$ is irredicible over  hence over $\mathbb{Q}_{2}$. Hence we have one invariant with $inv_{\mathfrak{p}_i} (End_k (A)\otimes \mathbb{Q}) \equiv 0 \mod \mathbb{Z}$  which shows the $\dim A=4$.
\item If $P(X)=X^8\pm \sqrt{3q}X^7+2qX^6\pm q \sqrt{3q}X^5+q^2X^4\pm q^2 \sqrt{3q}X^3+2q^3X^2\pm q^3\sqrt{3q}X+q^4$ where $p=3$, is irreducible over  $\frac{\mathbb{Z}}{2\mathbb{Z}}$, hence over $\mathbb{Q}$ with splitting field, $\mathbb{Q}(\zeta_{60})$. The similiar argument above shows that $P(X)$
 corresponds to abelian variety of dimension $4$.
\item  If $P(X)=X^8\pm \sqrt{5q}X^7+2qX^6\pm q \sqrt{5q}X^5+3q^2X^4\pm q^2 \sqrt{5q}X^3+2q^3X^2\pm q^3\sqrt{5q}X^+q^4$, where  $p=5$ is irreducible in $\frac{\mathbb{Z}}{3\mathbb{Z}}$, hence irreducible over $\mathbb{Q}$ with splitting field $\mathbb{Q}(\zeta_{15})$. But $P(X/\sqrt{q})$ is reducible over $\mathbb{Q}(\sqrt{5})[X]$ with $P(X/ \sqrt{q})=F_1(X/ \sqrt{q})F_2(X/ \sqrt{q})$ each irreducible over  $\mathbb{Q}(\sqrt{5})[X]$ with roots $\zeta_{30}$ and $\zeta_{15}$ respectively. But $[ \mathbb{Q}(\sqrt{5}):\mathbb{Q}(\zeta_{15})]=4$. Passing through completion we have $[ \mathbb{Q}_5(\sqrt{5}):\mathbb{Q}_5(\zeta_{15})]=4$ with root of $F_1(X/ \sqrt{q})$ and $F_2(X/ \sqrt{q})$ as $\zeta_{15}$ and $-\zeta_{15}$ respectively and since $\deg F_i=4$,  $F_i(X/ \sqrt{q})$ are irreducible over  $\mathbb{Q}_5(\sqrt{5})$. $P(X)=F_1(X)F_2(X)$. But $F_i$  have  coefficients from $\mathbb{Q}_5(\sqrt{5}) / \mathbb{Q}_5$. Hence $P(X)$ is irreducible over $\mathbb{Q}_5$. This implies $\dim A=4$.
\end{enumerate}
\end{proof}
\section{Dimension 5}
The characteristic polynomial of Frobenius of an abelian variety of dimension $5$ is given by
\begin{displaymath}
P(X)=X^{10}+a_{1}X^9+a_{2}X^8+a_{3}X^7+a_{4}X^6+a_5X^5+a_4qaX^4+a_{3}q^2X^3+a_{2}q^3X^2+a_{1}q^4X+q^5.
\end{displaymath}
If $P(X)$ is irreducible then we have following cases\\

\subsection{Case $a_{2i+1} \neq 0$} 
If $a_{2i+1} \neq 0$ then we have following cases,\\
\emph{case 1.} If $G(x)$ is irreducible then

\begin{eqnarray*}\nonumber 
H(x)=x^{20}+ \left( -{\frac {{{ a_1}}^{2}}{q}}+2{\frac {{ a_2}}{q}} \right) x^{18}+ \left( -2{\frac {{ a_3}{ a_1}}{{q}^{2}}}+
{\frac {{{ a_2}}^{2}}{{q}^{2}}}+2{\frac {{ a_4}}{{q}^{2}}}\right) x^{16}+\\
 \left( -2{\frac {{ a_5}{ a_1}}{{q}^{3}}}+2{\frac {{ a_4}{ a_2}}{{q}^{3}}}+2{\frac {{ a_4}}{{q}^{2}}}
-{\frac {{{ a_3}}^{2}}{{q}^{3}}} \right) x^{14}+ \left( 2{\frac{{ a_2}}{q}}+{\frac {{{ a_4}}^{2}}{{q}^{4}}}+2{\frac {{ a_4}
{ a_2}}{{q}^{3}}}-2{\frac {{ a_3}{ a_1}}{{q}^{2}}}-2{\frac {{ a_5}{ a_3}}{{q}^{4}}} \right) x^{12}+\\
 \left( -2{\frac {{{ a_1}}^{2}}{q}}+2{\frac {{{ a_4}}^{2}}{{q}^{4}}}-2{
\frac {{{ a_3}}^{2}}{{q}^{3}}}+2-{\frac {{{ a_5}}^{2}}{{q}^{5}}}+2{\frac {{{ a_2}}^{2}}{{q}^{2}}} \right) x^{10}+
 \left( 2{\frac {{ a_2}}{q}}+{\frac {{{ a_4}}^{2}}{{q}^{4}}}+2{\frac {{
 a_4}{ a_2}}{{q}^{3}}}-2{\frac {{ a_3}{ a_1}}{{q}^{2}}}-2{\frac {{ a_5}{ a_3}}{{q}^{4}}} \right) x^{8}\\
+ \left( -2{
\frac {{ a_5}{ a_1}}{{q}^{3}}}+2{\frac {{ a_4}{ a_2}}{{q
}^{3}}}+2{\frac {{ a_4}}{{q}^{2}}}-{\frac {{{ a_3}}^{2}}{{q}^{3}
}} \right) x^{6}+ \left( -2{\frac {{ a_3}{ a_1}}{{q}^{2}}}+{
\frac {{{ a_2}}^{2}}{{q}^{2}}}+2{\frac {{ a_4}}{{q}^{2}}}
 \right) x^{4}+ \left( -{\frac {{{ a_1}}^{2}}{q}}+2{\frac {{
a_2}}{q}} \right) x^{2}+1
\end{eqnarray*} 
whose roots are mth roots of unity where $\phi(m)=4g= 20$ which implies $m \in \{25, 33, 44, 50, 66\}$. Therefore we have\\
\begin{enumerate}
\item $x^{25}-1=(x-1)(1+x^4+x^3+x^2+x)(1+x^{20}+x^{15}+x^{10}+x^5)$
\item $x^{33}-1=(x-1)(1+x^{10}+x^9+x^8+x^7+x^6+x^5+x^4+x^3+x^2+x)(1+x^2+x)(1-x+x^3-x^4+x^6-x^7+x^9-x^{10}+x^{11}-x^{13}+x^{14}-x^{16}+x^{17}-x^{19}+x^{20})$
\item $x^{44}-1=(x-1)(1+x^{10}+x^9+x^8+x^7+x^6+x^5+x^4+x^3+x^2+x)(1+x)(1-x+x^2-x^3+x^4-x^5+x^6-x^7+x^8-x^9+x^{10})(1+x^2)(x^{20}-x^{18}+x^{16}-x^{14}+x^{12}-x^{10}+x^8-x^6+x^4-x^2+1)$
\item $x^{50}-1=(x-1)(1+x^4+x^3+x^2+x)(1+x^{20}+x^{15}+x^{10}+x^5)(1+x)(1-x+x^2-x^3+x^4)(1-x^5+x^{10}-x^{15}+x^{20})$
\item $x^{66}-1=(x-1)(1+x^{10}+x^9+x^8+x^7+x^6+x^5+x^4+x^3+x^2+x)(1+x^2+x)(1-x+x^3-x^4+x^6-x^7+x^9-x^{10}+x^{11}-x^{13}+x^{14}-x^{16}+x^{17}-x^{19}+x^{20})(1+x)(1-x+x^2-x^3+x^4-x^5+x^6-x^7+x^8-x^9+x^{10})(1-x+x^2)(1+x-x^3-x^4+x^6+x^7-x^9-x^{10}-x^{11}+x^{13}+x^{14}-x^{16}-x^{17}+x^{19}+x^{20}).$
\end{enumerate}
Let \\
$E_1:= -{\frac {{{ a_1}}^{2}}{q}}+2{\frac {{ a_2}}{q}}$\\
$E_2:=-2{\frac {{ a_3}{ a_1}}{{q}^{2}}}+{\frac {{{ a_2}}^{2}}{{q}^{2}}}+2{\frac {{ a_4}}{{q}^{2}}}$\\
$E_3:=-2{\frac {{ a_5}{ a_1}}{{q}^{3}}}+2{\frac {{ a_4}{ a_2}}{{q}^{3}}}+2{\frac {{ a_4}}{{q}^{2}}}-{\frac{{{a_3}}^{2}}{{q}^{3}}} $\\
$E_4:=2{\frac{{ a_2}}{q}}+{\frac {{{ a_4}}^{2}}{{q}^{4}}}+2{\frac {{ a_4}{ a_2}}{{q}^{3}}}-2{\frac {{ a_3}{ a_1}}{{q}^{2}}}-2{\frac {{ a_5}{ a_3}}{{q}^{4}}}$\\
$E_5= -2{\frac {{{ a_1}}^{2}}{q}}+2{\frac {{{ a_4}}^{2}}{{q}^{4}}}-2{\frac {{{ a_3}}^{2}}{{q}^{3}}}+2-{\frac {{{a_5}}^{2}}{{q}^{5}}} +2{\frac {{{ a_2}}^{2}}{{q}^{2}}}$.\\

Comparing $H(x)$ with degree $20$ irreducible cyclotomic factor we have only one possibility for $H(x)$
namely, $H(x)=x^{20}-x^{18}+x^{16}-x^{14}+x^{12}-x^{10}+x^8-x^6+x^4-x^2+1$ then $E_ 1=-1,~E_2=1,~E_3 =-1,~E_4 =1,~E_5=-1$ in which case we have following solutions for $a_i$'s
     \begin{enumerate}
     \item $a_1= \sqrt{11q}, ~a_2=5q, ~a_3=q\sqrt{11q},~a_4=-q^2,~ a_5=q^2\sqrt{11q}$ which is possible solution if q is odd power of 11.
     \item Also $a_1$ satisfies $165Z^2q^4+330Z^6q^2-55Z^8q+Z^{10}-11q^5-462Z^4q^3$ and  $14949Z^2q^4+1034Z^6q^2-55Z^8q+Z^{10}-11q^5-7502Z^4q^3$ and $11605Z^2q^4+858Z^6q^2-55Z^8q+Z^{10}-5819q^5-5214Z^4q^3$ which are irreducible if $p \neq 11$ by Eisenstein's criteria, hence has no integer solutions. If $p=11$ then $p=1 \mod 5 $, substituting $q=1$ still gives no integer solutions  for these polynomials, hence they have no integer roots.
     \end{enumerate}
\emph{Case 2.} If $G(x)$ is reducible. Then  by \ref{Red}, $G(x)=F_1(x).F_2(x)$.
 If $F_i \in \mathbb{Q}({\sqrt{q}})[x]\setminus \mathbb{Q}[x]$ then its roots are  mth root of unity  where $\phi(m)=2g=10$  where $m \in \{11,22\}$ in which case $F_i(x)F_i(x)^{\sigma}$ is a irreducible cyclotomic factor of $x^m-1$  of degree $10$ with even powers of $x$. 
We have following factorizations.
\begin{enumerate}
\item$ x^{11}-1=(x-1)(1+x^{10}+x^9+x^8+x^7+x^6+x^5+x^4+x^3+x^2+x)$
\item$ x^{22}-1=(x-1)(1+x^{10}+x^9+x^8+x^7+x^6+x^5+x^4+x^3+x^2+x)(1+x)(1-x+x^2-x^3+x^4-x^5+x^6-x^7+x^8-x^9+x^{10})$
\end{enumerate}

If $F_i \in \mathbb{Q}[x]$ then its roots are  mth root of unity  where  $\phi(m)=g=5$ for which there is no solutions.

 Now $H(x)=F_1 F_{1}^{\sigma}F_2 F_{2}^{\sigma}$ where not both $F_i \in \mathbb{Q}[x]$ as discussed earlier. Since $H(x)$ has only even degree terms, we look at all possibilities $F_1 F_{1}^{\sigma}F_2 F_{2}^{\sigma}$ from above and we just get one possibility namely,
 $H(x)=1+x^2+x^4+x^{6}+x^{8}+x^{10}+x^{12}+x^{14}+x^{16}+x^{18}+x^{20}$ in which case, all $E_i=1$ for $1\leq i\leq  5$, which gives $a_1$ satisfies $z^2-q$ which has no integer solution as q is odd power of $p$. The solutions of $a_1$ are root of $10949Z^2q^4+842Z^6q^2-53Z^8q+Z^{10}-7921q^5-4842Z^4q^3$ or $3221Z^2q^4+442Z^6q^2-37Z^8q+Z^{10}-529q^5-2074Z^4q^3$ or $565Z^2q^4+666Z^6q^2-53Z^8q+Z^{10}-q^5-2202Z^4q^3$ none of which has integer solutions by $\mod 3 \mod 5$ test.\\

\subsection{Case $a_{2i+1}=0$}
 If $a_{2i+1} =0$ then $G(x)=x^{10}+{\frac {{ a_2}x^{8}}{q}}+{\frac {{ a_4}x^{6}}{{q}^{2}}}+{\frac {{ a_4}x^{4}}{{q}^{2}}}+{\frac {{ a_2}x^{2}}{q}}+1$, with roots as mth root of unity where,
\begin{enumerate}
\item If $G(x)$ is irreducible then $\phi(m)=10$ which gives $m \in \{11,22\}$ and $G(x)$ is of degree $10$ and has only even degree terms but $x^{11}-1$ and $x^{22}-1$ have no irreducible factor with only even terms of degree $10$. So this is not possible.
\item  If $G(x)$ is reducible then it is product of two degree $5$ irreducible factors of $x^m-1$ above. But there are no degree $5$ factors so this is not possible.
    \end{enumerate}
\subsection{ $P(X)$ reducible}
If $P(X)$ is reducible then $P(X)= h(X)^e$ where $e|5$ , $e>1$ so $e=5$. In that case $h(X)=X^2+aX+q$ which by argument done for dimension $2$ does not corresponds to simple abelian variety.

\begin{theorem} The characteristic polynomial of the a simple supersingular abelian variety of dimension $5$ over $\mathbb{F}_q$ ($q=p^n$, $n$ odd) is given by\\
\begin{displaymath}
p=11: X^{10}\pm\sqrt{11q}X^9+5qX^8 \pm q\sqrt{11q}X^7-q^2X^6 \pm q^2 \sqrt{11q}X^5-q^3X^4\pm q^3\sqrt{11q}X^3+5q^4X^2\pm q^4\sqrt{11q}X+q^5.
\end{displaymath}
\end{theorem}
\begin{proof} From discussion above we have established   $P(X)= X^{10}\pm\sqrt{11q}X^9+5qX^8 \pm q\sqrt{11q}X^7-q^2X^6 \pm q^2 \sqrt{11q}X^5-q^3X^4\pm q^3\sqrt{11q}X^3+5q^4X^2\pm q^4\sqrt{11q}X+q^5$, $p =11$ is a supersingular  Weil polynomial of degree $10$. Also $P(X)$ irreducible in $\frac{\mathbb{Z}}{2\mathbb{Z}}$, hence irreducible over $\mathbb{Q}$.
$P(X/\sqrt{q})$ is irreducible over $\mathbb{Q}(\sqrt{11})[X]$ with one of the roots as $\zeta_{44}$ and splitting field $\mathbb{Q}(\zeta_{44})$. 
We have $[ \mathbb{Q}(\sqrt{11}):\mathbb{Q}(\zeta_{44})]=10$. Passing through completion  and using theorem \ref{padics} we get,\\
\begin{displaymath}
[\mathbb{Q}_{11}:\mathbb{Q}_{11}(\sqrt{11})][\mathbb{Q}_{11}(\sqrt{11}):\mathbb{Q}_{11}(\zeta_{44})]=[\mathbb{Q}_{11}:\mathbb{Q}_{11}(\zeta_{44})]=20.
\end{displaymath}
Therefore $[ \mathbb{Q}_{11}(\sqrt{11}):\mathbb{Q}_{11}(\zeta_{44})]=10$. But $P(X/\sqrt{q}) \in \mathbb{Q}_{11}(\sqrt{11})$ has degree $10$  and $\zeta_{44}$ as one of the  roots. 
This implies  $P(X/\sqrt{q})$ and hence $P(X)$ are irreducible over $\mathbb{Q}(\sqrt{11})$ which implies $P(X)$ is irredicible over $\mathbb{Q}_{11}$. Hence we have one invariant with $inv_{\mathfrak{p}_i} (End_k (A)\otimes \mathbb{Q}) \equiv 0 \mod \mathbb{Z}$  which shows the $\dim A=5$.
\end{proof}
\section{Dimension 6}
The characteristic polynomial of Frobenius of an abelian variety of dimension $6$ is given by
\begin{displaymath}
P(X)=X^{12}+a_{1}X^{11}+a_{2}X^{10}+a_{3}X^9+a_{4}X^8+a_5X^7+a_6X^6+qa_5X^5+q^2X^4+q^3a_{3}X^3+q^4a_{2}X^2+q^5a_{1}X+q^6.
\end{displaymath}
If $P(X)$ is irreducible  then we have following cases.\\
\subsection{ Case $a_{2i+1} \neq 0$}
If $a_{2i+1} \neq 0$ then we have following cases\\

\emph{Case 1.} If $G(x)$ is irreducible then\\

\begin{multline}\nonumber
H(x)={x}^{24}+ \left( 2{\frac {{ a_2}}{q}}-{\frac {{{ a_1}}^{2}}{q}}
 \right) {x}^{22}+ \left( -2{\frac {{ a_3}{ a_1}}{{q}^{2}}}+{
\frac {{{ a_2}}^{2}}{{q}^{2}}}+2{\frac {{ a_4}}{{q}^{2}}}
 \right) {x}^{20}+ \left( 2{\frac {{ a_6}}{{q}^{3}}}-{\frac {{{
 a_3}}^{2}}{{q}^{3}}}-2{\frac {{ a_5}{ a_1}}{{q}^{3}}}+2{
\frac {{ a_4}{ a_2}}{{q}^{3}}} \right) {x}^{18}+\\ \left( -2{
\frac {{ a_5}{ a_3}}{{q}^{4}}}+2{\frac {{ a_4}}{{q}^{2}}}+2
{\frac {{ a_6}{ a_2}}{{q}^{4}}}-2{\frac {{ a_5}{ a_1}}
{{q}^{3}}}+{\frac {{{ a_4}}^{2}}{{q}^{4}}} \right) {x}^{16}+ \left(
2{\frac {{ a_2}}{q}}-2{\frac {{ a_3}{ a_1}}{{q}^{2}}}-2{
\frac {{ a_5}{ a_3}}{{q}^{4}}}-{\frac {{{ a_5}}^{2}}{{q}^{5}}}
+2{\frac {{ a_4}{ a_2}}{{q}^{3}}}+2{\frac {{ a_6}{ a_4
}}{{q}^{5}}} \right) {x}^{14}+\\
 \left( 2{\frac {{{ a_2}}^{2}}{{q}^{
2}}}+{\frac {{{ a_6}}^{2}}{{q}^{6}}}+2{\frac {{{ a_4}}^{2}}{{q}^
{4}}}-2{\frac {{{ a_5}}^{2}}{{q}^{5}}}-2{\frac {{{ a_1}}^{2}}{
q}}+2-2{\frac {{{ a_3}}^{2}}{{q}^{3}}} \right) {x}^{12}+ \left( 2
{\frac {{ a_2}}{q}}-2{\frac {{ a_3}{ a_1}}{{q}^{2}}}-2{
\frac {{ a_5}{ a_3}}{{q}^{4}}}-{\frac {{{ a_5}}^{2}}{{q}^{5}}}
+2{\frac {{ a_4}{ a_2}}{{q}^{3}}}+2{\frac {{ a_6}{ a_4
}}{{q}^{5}}} \right) {x}^{10}+\\ \left( -2{\frac {{ a_5}{ a_3}}{
{q}^{4}}}+2{\frac {{ a_4}}{{q}^{2}}}+2{\frac {{ a_6}{ a_2}
}{{q}^{4}}}-2{\frac {{ a_5}{ a_1}}{{q}^{3}}}+{\frac {{{ a_4}
}^{2}}{{q}^{4}}} \right) {x}^{8}+ \left( 2{\frac {{ a_6}}{{q}^{3}}
}-{\frac {{{ a_3}}^{2}}{{q}^{3}}}-2{\frac {{ a_5}{ a_1}}{{q}
^{3}}}+2{\frac {{ a_4}{ a_2}}{{q}^{3}}} \right) {x}^{6}+\\
 \left( -2{\frac {{ a_3}{ a_1}}{{q}^{2}}}+{\frac {{{ a_2}}^{
2}}{{q}^{2}}}+2{\frac {{ a_4}}{{q}^{2}}} \right) {x}^{4}+
 \left( 2{\frac {{ a_2}}{q}}-{\frac {{{ a_1}}^{2}}{q}} \right) {x}^{2}+1
\end{multline}

whose roots are mth roots of unity where $\phi=4g= 24$ which implies $m \in \{35, 39, 45, 52, 56, 70, 72, 78, 84, 90\}$. Each of which has following factorization.
\begin{enumerate}
\item $x^{35}-1 = (x-1)(1+x^6+x^5+x^4+x^3+x^2+x)(1+x^4+x^3+x^2+x)(1-x+x^5-x^6+x^7-x^8+x^{10}-x^{11}+x^{12}-x^{13}+x^{14}-x^{16}+x^{17}-x^{18}+x^{19}-x^{23}+x^{24})$
\item $x^{39} -1 =(x-1)(1+x^{12}+x^{11}+x^{10}+x^9+x^8+x^7+x^6+x^5+x^4+x^3+x^2+x)(1+x^2+x)(1-x+x^3-x^4+x^6-x^7+x^9-x^{10}+x^{12}-x^{14}+x^{15}-x^{17}+x^{18}-x^{20}+x^{21}-x^{23}+x^{24})$

\item $x^{45}-1=(x-1)(1+x^4+x^3+x^2+x)(1+x^2+x)(1-x+x^3-x^4+x^5-x^7+x^8)(x^6+x^3+1)(x^{24}-x^{21}+x^{15}-x^{12}+x^9-x^3+1)$
\item $x^{52}-1 =(x-1)(1+x^{12}+x^{11}+x^{10}+x^9+x^8+x^7+x^6+x^5+x^4+x^3+x^2+x)(1+x)(1-x+x^2-x^3+x^4-x^5+x^6-x^7+x^8-x^9+x^{10}-x^{11}+x^{12})(1+x^2)
    (x^{24}-x^{22}+x^{20}-x^{18}+x^{16}-x^{14}+x^{12}-x^{10}+x^8-x^6+x^4-x^2+1)$
\item $x^{56}-1=(x-1)(1+x^6+x^5+x^4+x^3+x^2+x)(1+x)(1-x+x^2-x^3+x^4-x^5+x^6)(1+x^2)(x^{12}-x^{10}+x^8-x^6+x^4-x^2+1)(1+x^4)(x^{24}-x^{20}+x^{16}-x^{12}+x^8-x^4+1)$
\item $x^{70}-1=(x-1)(1+x^6+x^5+x^4+x^3+x^2+x)(1+x^4+x^3+x^2+x)(1-x+x^5-x^6+x^7-x^8+x^{10}-x^{11}+x^{12}-x^{13}+x^{14}-x^{16}+x^{17}-x^{18}+x^{19}-x^{23}+x^{24})
    (1+x)(1-x+x^2-x^3+x^4-x^5+x^6)(1-x+x^2-x^3+x^4)
    (1+x-x^5-x^6-x^7-x^8+x^{10}+x^{11}+x^{12}+x^{13}+x^{14}-x^{16}-x^{17}-x^{18}-x^{19}+x^{23}+x^{24})$
\item $x^{72}-1=(x-1)(1+x^2+x)(x^6+x^3+1)(1+x)(1-x+x^2)(1-x^3+x^6)(1+x^2)(x^4-x^2+1)(x^{12}-x^6+1)(1+x^4)(x^8-x^4+1)(x^{24}-x^{12}+1)$
\item $x^{78}-1=(x-1)(1+x^{12}+x^{11}+x^{10}+x^9+x^8+x^7+x^6+x^5+x^4+x^3+x^2+x)(1+x^2+x)(1-x+x^3-x^4+x^6-x^7+x^9-x^{10}+x^{12}-x^{14}+x^{15}-x^{17}+x^{18}-x^{20}+x^{21}-x^{23}+x^{24})
    (1+x)
    (1-x+x^2-x^3+x^4-x^5+x^6-x^7+x^8-x^9+x^{10}-x^{11}+x^{12})(1-x+x^2)
    (1+x-x^3-x^4+x^6+x^7-x^9-x^{10}+x^{12}-x^{14}-x^{15}+x^{17}+x^{18}-x^{20}-x^{21}+x^{23}+x^{24})$
\item $x^{84}-1=(x-1)(1+x^6+x^5+x^4+x^3+x^2+x)(1+x^2+x)(1-x+x^3-x^4+x^6-x^8+x^9-x^{11}+x^{12})(1+x)(1-x+x^2-x^3+x^4-x^5+x^6)(1-x+x^2)(1+x-x^3-x^4+x^6-x^8-x^9+x^{11}+x^{12})(1+x^2)
    (x^{12}-x^{10}+x^8-x^6+x^4-x^2+1)
    (x^4-x^2+1)(x^{24}+x^{22}-x^{18}-x^{16}+x^{12}-x^8-x^6+x^2+1)$
\item $x^{90}-1=(x-1)(1+x^4+x^3+x^2+x)(1+x^2+x)(1-x+x^3-x^4+x^5-x^7+x^8)(x^6+x^3+1)(x^{24}-x^{21}+x^{15}-x^{12}+x^9-x^3+1)(1+x)(1-x+x^2-x^3+x^4)(1-x+x^2)(1+x-x^3-x^4-x^5+x^7+x^8)(1-x^3+x^6)(x^{24}+x^{21}-x^{15}-x^{12}-x^9+x^3+1).$
\end{enumerate}
Let\\
$E_1:= 2{\frac {{ a_2}}{q}}-\frac {{{ a_1}}^{2}}{q}$\\
$E_2:= -2{\frac {{ a_3}{ a_1}}{{q}^{2}}}+{\frac {{{ a_2}}^{2}}{{q}^{2}}}+2{\frac {{ a_4}}{{q}^{2}}}$\\
$E_3:= 2{\frac {{ a_6}}{{q}^{3}}}-{\frac {{{a_3}}^{2}}{{q}^{3}}}-2{\frac {{ a_5}{ a_1}}{{q}^{3}}}+2{\frac {{ a_4}{ a_2}}{{q}^{3}}}$ \\
$E_4:= -2{\frac {{ a_5}{ a_3}}{{q}^{4}}}+2{\frac {{ a_4}}{{q}^{2}}}+2{\frac {{ a_6}{ a_2}}{{q}^{4}}}-2{\frac {{ a_5}{ a_1}}{{q}^{3}}}+{\frac {{{ a_4}}^{2}}{{q}^{4}}}$ \\
$E_5:=2{\frac {{ a_2}}{q}}-2{\frac {{ a_3}{ a_1}}{{q}^{2}}}-2{\frac {{ a_5}{ a_3}}{{q}^{4}}}-{\frac {{{ a_5}}^{2}}{{q}^{5}}}+2{\frac {{ a_4}{ a_2}}{{q}^{3}}}+2{\frac {{ a_6}{ a_4}}{{q}^{5}}}$ \\
$E_6:= 2{\frac {{{ a_2}}^{2}}{{q}^{2}}}+{\frac {{{ a_6}}^{2}}{{q}^{6}}}+2{\frac {{{ a_4}}^{2}}{{q}^{4}}}-2{\frac {{{ a_5}}^{2}}{{q}^{5}}}-2{\frac {{{ a_1}}^{2}}{
q}}+2-2{\frac {{{ a_3}}^{2}}{{q}^{3}}}$\\

Comparing $H(x)$ with degree $24$ irreducible cyclotomic factor we have following possibilities for $H(x)$

 \begin{enumerate}
\item  If $H(x)=x^{24}-x^{22}+x^{20}-x^{18}+x^{16}-x^{14}+x^{12}-x^{10}+x^8-x^6+x^4-x^2+1$ then $E_1=-1 ,~E_2=1 ,~E_3=-1 ,~E_4=1 ,~E_5=-1,~E_6=1 $ which gives $a_1$ satisfies
\begin{enumerate}
\item $117q^2-26Z^2q+Z^4$
\item $3q^6-286Z^2q^5-78Z^{10}q+Z^{12}+1287q^4Z^4+715q^2Z^8-1716q^3Z^6$
\item $13q^6-52702Z^2q^5-78Z^{10}q+Z^{12}+43719q^4Z^4+1547q^2Z^8-12532q^3Z^6$
\item $8125q^6-64350Z^2q^5-78Z^{10}q+Z^{12}+116519q^4Z^4+2171q^2Z^8-25844q^3Z^6$
\item $81133q^6-138398Z^2q^5-78Z^{10}q+Z^{12}+81991q^4Z^4+1963q^2Z^8-20020q^3Z^6$
\item $137917q^6-244062Z^2q^5-78Z^{10}q+Z^{12}+126503q^4Z^4+2171q^2Z^8-25844q^3Z^6$
\end{enumerate}
none of which has an integer solution by $\mod 3,\mod 5$ test.
\item  If $H(x)=x^{24}-x^{20}+x^{16}-x^{12}+x^8-x^4+1$ then $E_1= 0,~E_2= -1,~E_3= 0,~E_4=1,~E_5=0,~E_6= -1$  which gives $a_i$'s satisfies $a_1=\pm \sqrt{2q},~ a_2=q,~a_3=0,~a_4=-q^2,~a_5=-q^2a_1,~a_6=-q^3$. The other solutions for $a_1$ are
    \begin{enumerate}
    \item $-14q+Z^2$
\item $ -8q^3+332Z^2q^2-38Z^4q+Z^6$
\item $ -56q^3+140Z^2q^2-42Z^4q+Z^6$
\item $64q^6-7488Z^2q^5-52Z^{10}q+Z^{12}+12016q^4Z^4+828q^2Z^8-5088q^3Z^6$
\item $64q^6-1728Z^2q^5-76Z^{10}q+Z^{12}+10480q^4Z^4+1212q^2Z^8-6432q^3Z^6$
\item $3136 q^6-40768 Z^2 q^5-84 Z^{10} q+Z^{12}+49392 q^4 Z^4+2044 q^2 Z^8-18144 q^3 Z^6$
\item $817216 q^6-632512 Z^2 q^5-76 Z^{10} q+Z^{12}+189680 q^4 Z^4+2108 q^2 Z^8-27936 q^3 Z^6$
    \end{enumerate}
none of which has integer solution by $\mod 3,\mod 5$ test.
\item  If $H(x)=x^{24}-x^{12}+1$ then $E_1=0 ,~E_2=0 ,~E_3=0 ,~E_4= 0,~E_5=0 ,~E_6=-1 $ then solutions are
 \begin{enumerate}
 \item $a_1=a_2=a_4=a_5=0,~ a_3= \pm q \sqrt{2q},~ a_6=q^3$  which is possible for q is odd power of 2.
 \item $a_3$ satisfies $Z^2-6q$ which has no integer solution for any q.
 \item The other solutions of $a_1$ are roots of polynomials
 \begin{enumerate}
 \item $-8q^3+36Z^2q^2-12Z^4q+Z^6$
\item $ 2(-24q^3+36Z^2q^2-12Z^4q+Z^6)$
\item $ 2(9q^6-108Z^2q^5-24Z^{10}q+Z^{12}+333q^4Z^4+162q^2Z^8-372q^3Z^6)$
\item $ 2(Z^2-q)$
\item $ 2(5Z^2-2q)$
\item $ 2(Z^{12}-12Z^{10}q+54q^2Z^8-112q^3Z^6+105q^4Z^4-36Z^2q^5+q^6)$
   \end{enumerate}
   none of which has integer solutions by $\mod 3, \mod 5$ test.
  \end{enumerate}

\item  If $H(x)=x^{24}+x^{22}-x^{18}-x^{16}+x^{12}-x^8-x^6+x^2+1$ then $E_1= 1,~E_2= 0,~E_3=-1 ,~E_4= -1,~E_5=0,~E_6= 1$ then solutions are
    \begin{enumerate}
    \item $a_1=\pm \sqrt{3q},~ a_2=2q,~a_3=\pm q\sqrt{3q},~a_4=q^2,~a_5=0,~a_6=-q^3$ which is a possibility  if $q$ odd power of $3$.
    \item $a_1=\pm \sqrt{7q},~ a_2=4q,~a_3=\pm q \sqrt{7q},~a_4=-q^2,~a_5=-2q^2a_1,~a_6=-7q^3$ which is a possibility  for $q$ odd power of $7$.
    \end{enumerate}
    The other solutions of $a_1$ are
    \begin{enumerate}
    \item $-7q^3+35Z^2q^2-21Z^4q+Z^6$
    \item $-27q^3+747Z^2q^2-57Z^4q+Z^6$
    \item $q^6-186Z^2q^5-58Z^{10}q+Z^{12}+1423q^4Z^4+655q^2Z^8-1772q^3Z^6$
    \item $1849q^6-9682Z^2q^5-50Z^{10}q+Z^{12}+11775q^4Z^4+743q^2Z^8-4572q^3Z^6$
    \item $27889q^6-84410Z^2q^5-58Z^{10}q+Z^{12}+50927q^4Z^4+1215q^2Z^8-11628q^3Z^6$
    \item $337561q^6-707266Z^2q^5-98Z^{10}q+Z^{12}+287679q^4Z^4+3143q^2Z^8-44604q^3Z^6$
    \end{enumerate}
    none of which has integer solutions by $\mod 3,\mod 5$ test.
\end{enumerate}

Case 2.  If $G(x)$ is reducible then  by theorem \ref{Red}, we get $G(x)=F_1(x)F_2(x)$.
 If $F_i \in \mathbb{Q}({\sqrt{q}})[x] \setminus \mathbb{Q}[x]$ then its roots are  mth root of unity  where $\phi(m)=2g=12$  where $m \in \{13, 21, 26, 28, 36, 42 \}$ in which case $F_i(x)F_i(x)^{\sigma}$ is irreducible cyclotomic factor of $x^m-1$ of degree $12$.

If $F_i \in \mathbb{Q}[x]$ then its roots are  mth root of unity  where  $\phi(m)=g=6$ in which case $m \in \{ 7, 9, 14, 18\}$ and in that case $F_i$ is irreducible cyclotomic factor of $x^m-1$ of degree $6$. 
 Now $H(x)=F_1 F_{1}^{\sigma}F_2 F_{2}^{\sigma}$ where not both $F_i \in \mathbb{Q}[x]$ as discussed earlier. Since $H(x)$ has only even degree terms, we look at all possibilities $F_1 F_{1}^{\sigma}F_2 F_{2}^{\sigma}$ from above, all of them are listed  below.
 \begin{enumerate}
 \item $H(x)=1+{x}^{6}+{x}^{4}+{x}^{2}+{x}^{10}+{x}^{12}+{x}^{8}+{x}^{22}+{x}^{20}+{x}^{18}+{x}^{16}+{x}^{24}+{x}^{14}$ which gives $E_1=E_2=E_3=E_4=E_5=E_6=1$ then $a_1=\pm \sqrt{13q}, ~a_2=7q, a_3=\pm 3q\sqrt{13q}, ~a_4=15q^2, a_5= \pm 5q^2\sqrt{13q}, a_6=19q^3$ which is possible if $q$ is odd power of $13$.
The other solutions for $a_1$ are roots of following polynomials
\begin{enumerate}
\item $ -325q^3+299Z^2q^2-39Z^4q+Z^6$
\item $ -625q^3+339Z^2q^2-35Z^4q+Z^6$
\item $ q^6-262Z^2q^5-70Z^{10}q+Z^{12}+3919q^4Z^4+1487q^2Z^8-9172q^3Z^6$
\item $ 2809q^6-18766Z^2q^5-46Z^{10}q+Z^{12}+16447q^4Z^4+743q^2Z^8-5284q^3Z^6$
\item $ 169q^6-10478Z^2q^5-78Z^{10}q+Z^{12}+22815q^4Z^4+1495q^2Z^8-9828q^3Z^6$
\item $ 10609q^6-35206Z^2q^5-70Z^{10}q+Z^{12}+28463q^4Z^4+1279q^2Z^8-9172q^3Z^6$
\item $ Z^2-q$
\end{enumerate}
which has no solutions for any $q$, odd power of prime by $\mod 3 \mod 5$ test.
 \item $H(x)= 1+{x}^{6}-{x}^{2}+{x}^{12}-{x}^{8}-{x}^{22}+{x}^{18}-{x}^{16}+{x}^{24}$ then $E_1=E_4=-1,~E_2=E_5=0, ~E_3=E_6=0$. Then possibilities of $a_1$ are

     \begin{enumerate}
\item $ Z^2-q$
\item $-21q+Z^2$
\item $-q^3+83Z^2q^2-19Z^4q+Z^6$
\item $-189q^3+315Z^2q^2-63Z^4q+Z^6$
\item $49q^6-1862Z^2q^5-70Z^{10}q+Z^{12}+14063q^4Z^4+1407q^2Z^8-9492q^3Z^6$
\item $1681q^6-12022Z^2q^5-54Z^{10}q+Z^{12}+20911q^4Z^4+991q^2Z^8-7412q^3Z^6$
\item $6889q^6-33150Z^2q^5-94Z^{10}q+Z^{12}+46687q^4Z^4+2263q^2Z^8-18500q^3Z^6$
\item $1194649q^6-1387902Z^2q^5-94Z^{10}q+Z^{12}+419647q^4Z^4+3271q^2Z^8-53444q^3Z^6$
\end{enumerate}
none of which has integer solutions.
 \item $H(x)=1-4\,{x}^{6}+3\,{x}^{4}-2\,{x}^{2}-6\,{x}^{10}+7\,{x}^{12}+5\,{x}^{8}-2\,{x}^{22}+3\,{x}^{20}-4\,{x}^{18}+5\,{x}^{16}+{x}^{24}-6\,{x}^{14}$ then $E_1=-2, ~E_2=3,~ E_3=-4,~ E_4=5,~ E_5=-6,~ E_6=7$ then solutions are
     \begin{enumerate}
    \item  $a_1=0,~a_2=-q, ~ a_3=0,~ a_4=q^2,~ a_5=0,~ a_6=-q^3$
    \item  $a_1=\pm 2 \sqrt{7q},~a_2=13q,~ a_3=\pm 8q \sqrt{7q}, ~a_4=29q^2, ~a_5=\pm 14q^2 \sqrt{7q},~a_6=41q^3$.
    \end{enumerate}
    In other solutions $a_1$ are
    \begin{enumerate}
\item $-7q^3+35Z^2q^2-21Z^4q+Z^6$
\item $-7q^3+49Z^2q^2-14Z^4q+Z^6$
\item $-7q^3+21Z^2q^2-14Z^4q+Z^6$
\item $2(-7q^3+14Z^2q^2-7Z^4q+Z^6)$
\end{enumerate}
none of which has integer solutions by $\mod 3 \mod 5$ test.
 \item $H(x)={x}^{24}-{x}^{22}+{x}^{20}-2{x}^{18}+2{x}^{16}-2{x}^{14}+3{x}^{12}-2{x}^{10}+2{x}^{8}-2{x}^{6}+{x}^{4}-{x}^{2}+1$ which gives $E_1=-1, ~E_2=1, ~E_3=E_5=-2, ~E_4=2, ~E_6=3$ which implies $a_1$ is root of following polynomials.
\begin{enumerate}
\item $289q^6-52314Z^2q^5-90Z^{10}q+Z^{12}+69327q^4Z^4+2607q^2Z^8-26924q^3Z^6$

\item $Z^{36}-270qZ^{34}+31449q^2Z^{32}-2101968q^3Z^{30}+90425076q^4Z^{28}-2660312328q^5Z^{26}\\
      +55467987076Z^{24}q^6-837266714544Z^{22}q^7+9254863738350Z^{20}q^8-75180384842708Z^{18}q^9+\\447103601802606Z^{16}q^{10}- 1923833859293616Z^{14}q^{11}+5861411718291844Z^{12}q^{12}-\\ 12204967208661768Z^{10}q^{13}+16400387853252084Z^8q^{14}-12905852655480016Z^6q^{15}+\\
      4963241293734297Z^4q^{16}-620389337072142Z^2q^{17}+7265822679361q^{18}$
\item $-7q^3+35Z^2q^2-21Z^4q+Z^6$
\end{enumerate}
 none of which has integer roots by $\mod 3 \mod 5$ test.

\item $H(x)=x^{24}-2x^{18}+3x^{12}-2x^6+1$ which means $E_1=E_2=E_4=E_5=0$, $~E_3=-2$ and $E_6=3$ which has following solutions for $a_i$'s

\begin{enumerate}
\item $a_1=a_2=a_4=a_5=0,~a_3= \pm 2q \sqrt{3q},~a_6=5q^3$ which is possible if $q$ is odd power of $3$.
\item $a_1=a_2=a_3=a_4=a_5=0$ and $a_6=-q^3$, which is possibility.
\item The other solutions of $a_1$ are
      \begin{enumerate}
      \item $2(-3q^3+9Z^2q^2-6Z^4q+Z^6)$
      \item $2(-3q^3+9Z^2q^2-6Z^4q+Z^6)$
      \item $2(-3q^3+9Z^2q^2-6Z^4q+Z^6)$
      \item $2(-27q^3+81Z^2q^2-18Z^4q+Z^6)$
      \item $4(-3q^3+9Z^2q^2-6Z^4q+Z^6)$
      \end{enumerate}
      none of which has an integer solution by $\mod 3 \mod 5$ test hence not possible.
 \end{enumerate}
 \end{enumerate}
\subsection{Case $a_{2i+1}=0$}
If  $a_{2i+1}=0$  for all $i$. Then  $G(x)={x}^{12}+{\frac {{a_2}{x}^{10}}{q}}+{\frac {{a_4}{x}^{8}}{{q}^{2}}}+{\frac {{a_6}{x}^{6}}{{q}^{3}}}+{\frac {{a_4}{x}^{4}}{{q}^{2}}}+{\frac {{a_2}{x}^{2}}{q}}+1$. We have following cases.
\begin{enumerate}
 \item If $G(x)$ is irreducible over $\mathbb{Q}$ then it is a degree $12$ irreducible cyclotomic factor of $x^m-1$ where $\phi(m)=12$. Since $G(x)$ has only even degree terms we have following possibilities
\begin{enumerate}
\item $G(x)=x^{12}-x^{10}+x^8-x^6+x^4-x^2+1$ which gives $X^{12} - qX^{10}+q^2X^8 - q^3X^6+q^4X^4-q^5X^2+q^6$  which is a possibility for $P(X)$ for $p \neq 7$. For $p=7$, $P(X)$ is reducible hence not possible.
\item $G(x)= x^{12}-x^6+1$ which gives  $P(X)= X^{12}-q^3X^6+q^6$ which is irreducible if $p \neq 3$ hence is possibility. For $p=3$ we have $X^{12}-q^3X^6+q^6=(X^6-\sqrt{3q^3}X^3+q^3)(X^6+\sqrt{3q^3}X^3+q^3)$ which is not irreducible hence is not a possibility.
\end{enumerate}
\item If $G(x)$ is reducible over $\mathbb{Q}$ then $G(x)=F_1F_2$ where $F_i$ are $6$ irreducible cyclotomic factor of $x^m-1$ where $\phi(m)=6$. Since $G(x)$ has only even degree terms we have following possibilities,
\begin{enumerate}
\item $G(x)=x^{12}+x^{10}+x^8+x^6+x^4+x^2+1$ which gives   $X^{12}+ qX^{10}+q^2X^8 + q^3X^6+q^4X^4 + q^5X^2+q^6$  which is a possibility for $P(X)$.
\item $G(x)= x^{12}+x^6+1$ which gives $P(X)= X^{12}+q^3X^6+q^6$ which is a possibility.
\end{enumerate}
\end{enumerate}
\subsection{$P(X)$  reducible}
 $P(X)$ is reducible then $P(X)= h(X)^e$ where $e|g$ , $e>1$ and $h(X)\in \mathbb{Z}[X]$. So $e =2, ~3$ or  $6$. Then,
 \begin{enumerate}
 \item If $e=2$ then  $\deg (h(X))=6$ say $h(X)=(X^6+fX^5+aX^4+bX^3+cX^2+dX \pm q^3)$. Then we have $G(t) =\left( {t}^{6}+{\frac {a{t}^{4}}{q}}+{\frac {c{t}^{2}}{{q}^{2}}}+1\right)-\frac{1}{\sqrt{q}} \left( f{t}^{5}+{\frac {b{t}^{3}}{q}}+{\frac {dt}{{q}^{2}}} \right)$. 
\begin{enumerate} 
\item If constant term of $h(X)$ has minus sign, then there is cyclotomic factor to compare with $H(t)$.
\item If constant term of $h(X)$ has plus sign, then $h(X)$ corresponds to characteristic polynomial of dimension  $3$ supersingular abelian variety and since all of them appear see \cite{Nart}, by Tate's theorem the abelian variety corresponding to $P(X)$ is not simple. 
\end{enumerate}

 \item If $e=3$ then  $\deg (h(X))=4$ say  $\deg (h(X))=4$ say $h(X)=(X^4+bX^3+cX^2+dX \pm g)$. This case is already discussed in dimension $4$, $P(X)$ reducible case.
\end{enumerate}
Hence we have,
\begin{theorem}The characteristic polynomial of a supersingular abelian variety over $\mathbb{F}_q$ , $q=p^n, ~ n$ odd, of dimension $6$ is given by one of following polynomials.
\begin{enumerate}
\item $p=2: X^{12} \pm \sqrt{2q}X^{11}+qX^{10}-q^2X^8-q^2( \pm \sqrt{2q})X^7-q^3X^6-q^3(\pm \sqrt{2q})X^5-q^4X^4+q^5X^2 \pm q^5\sqrt{2q}X+q^6$
\item $p=2: X^{12} \pm q \sqrt{2q}X^9+q^3X^6 \pm q^4\sqrt{2q}X^3+q^6$
\item $p=3:X^{12} \pm \sqrt{3q}X^{11}+2qX^{10} \pm 3q\sqrt{3q} X^9+q^2X^8-q^3X^6+q^4X^4 \pm 3q^4\sqrt{3q}X^3+2q^5X^2 \pm q^5\sqrt{3q}X+q^6$
\item $p=7:X^{12} \pm\sqrt{7q}X^{11}+4qX^{10} \pm q\sqrt{7q}X^9-q^2X^8-2q^2(\pm \sqrt{7q})X^7-7q^3X^6-2q^3( \pm\sqrt{7q})X^5-q^4X^4 \pm q^4\sqrt{7q}X^3+4q^5X^2 \pm q^5\sqrt{7q}X+q^6$
\item $p=7:X^{12} \pm 2\sqrt{7q}X^{11}+13qX^{10} \pm 8q\sqrt{7q}X^9+29q^2X^8 \pm14q^2 \sqrt{7q}X^7+41q^3X^6\pm14q^3 \sqrt{7q}X^5+29q^4X^4 \pm 8q^4\sqrt{7q}X^3+13q^5X^2\pm2q^5\sqrt{7q}X+q^6$
\item $p=13: X^{12} \pm \sqrt{13q}X^{11}+7qX^{10} \pm 3q \sqrt{13q}X^9+15q^2X^8 \pm 5q^2 \sqrt{13q}X^7+19q^3X^6 \pm 5q^3 \sqrt{13q}X^5+15q^4X^4 \pm 3q^4 \sqrt{13q}X^3+7q^5X^2 \pm q^5 \sqrt{13q}X+q^6$
\item  $X^{12} + qX^{10}+q^2X^8 + q^3X^6+q^4X^4 + q^5X^2+q^6$
\item $p \neq 7: X^{12} - qX^{10}+q^2X^8 - q^3X^6+q^4X^4 - q^5X^2+q^6$
\item $p \neq 3 :  X^{12}-q^3X^6+q^6$
\item $ X^{12}+q^3X^6+q^6$
\end{enumerate}
\end{theorem}
\begin{theorem}
 All of the polynomials listed above occur as characteristic polyonomial of Frobenuis of abelian varieties of dimension 6.
\end{theorem}

\begin{proof}
\begin{enumerate}
\item If $P(X)=X^{12} \pm \sqrt{2q}X^{11}+qX^{10}-q^2X^8-q^2( \pm \sqrt{2q})X^7-q^3X^6-q^3(\pm \sqrt{2q})X^5-q^4X^4+q^5X^2 \pm q^5\sqrt{2q}X+q^6$ for $p=2$, then $P(X)$ is irreducible in $\frac{\mathbb{Z}}{3\mathbb{Z}}$, hence irreducible over $\mathbb{Q}$.
$P(X/\sqrt{q})$ is irreducible over $\mathbb{Q}(\sqrt{2})[X]$ with one of the roots as $\zeta_{52}$ and splitting field $\mathbb{Q}(\zeta_{52})$. 
We have $[ \mathbb{Q}(\sqrt{2}):\mathbb{Q}(\zeta_{52})]=24$. Passing through completion  and using theorem \ref{padics} we get,\\
\begin{displaymath}
[\mathbb{Q}_{2}:\mathbb{Q}_{2}(\sqrt{2})][\mathbb{Q}_{2}(\sqrt{2}):\mathbb{Q}_{2}(\zeta_{52})]=[\mathbb{Q}_{2}:\mathbb{Q}_{2}(\zeta_{52})]=24.
\end{displaymath}
Therefore $[ \mathbb{Q}_2(\sqrt{2}):\mathbb{Q}_2(\zeta_{52})]=12$. But $P(X/\sqrt{q}) \in \mathbb{Q}_{2}(\sqrt{2})$ has degree $12$  and $\zeta_{52}$ as one of the  roots. 
This implies  $P(X/\sqrt{q})$ and hence $P(X)$ are irreducible over $\mathbb{Q}(\sqrt{2})$ which implies $P(X)$ is irredicible over  hence over $\mathbb{Q}_{2}$. Hence we have one invariant with $inv_{\mathfrak{p}_i} (End_k (A)\otimes \mathbb{Q}) \equiv 0 \mod \mathbb{Z}$  which shows the $\dim A=6$.

\item If $ P(X)=X^{12} \pm q \sqrt{2q}X^9+q^3X^6 \pm q^4\sqrt{2q}X^3+q^6$  where $p=2$, is irreducible over  $\frac{\mathbb{Z}}{3\mathbb{Z}}$, hence over $\mathbb{Q}$ with splitting field, $\mathbb{Q}(\zeta_{72})$. The similiar argument above shows that $P(X)$
 corresponds to abelian variety of dimension $6$.
\item If $P(X)=X^{12} \pm \sqrt{3q}X^{11}+2qX^{10} \pm 3q\sqrt{3q} X^9+q^2X^8-q^3X^6+q^4X^4 \pm 3q^4\sqrt{3q}X^3+2q^5X^2 \pm q^5\sqrt{3q}X+q^6$  where $p=3$, or $P(X)=X^{12} \pm\sqrt{7q}X^{11}+4qX^{10} \pm q\sqrt{7q}X^9-q^2X^8-2q^2(\pm \sqrt{7q})X^7-7q^3X^6-2q^3( \pm\sqrt{7q})X^5-q^4X^4 \pm q^4\sqrt{7q}X^3+4q^5X^2 \pm q^5\sqrt{7q}X+q^6$ where $p=7$ are irreducible over  $\frac{\mathbb{Z}}{2\mathbb{Z}}$, hence over $\mathbb{Q}$ with splitting field, $\mathbb{Q}(\zeta_{84})$. The similiar argument above shows that these $P(X)$ correspond to abelian variety of dimension $6$.

\item $P(X)=X^{12} \pm 2\sqrt{7q}X^{11}+13qX^{10} \pm 8q\sqrt{7q}X^9+29q^2X^8 \pm14q^2 \sqrt{7q}X^7+41q^3X^6\pm14q^3 \sqrt{7q}X^5+29q^4X^4 \pm 8q^4\sqrt{7q}X^3+13q^5X^2\pm2q^5\sqrt{7q}X+q^6$ where  $p=7$ is irreducible in $\frac{\mathbb{Z}}{2\mathbb{Z}}$, hence irreducible over $\mathbb{Q}$ with splitting field $\mathbb{Q}(\zeta_{28})$ . But $P(X/\sqrt{q})$ is reducible over $\mathbb{Q}(\sqrt{7})[X]$ with $P(X/ \sqrt{q})=F_1(X/ \sqrt{q})F_2(X/ \sqrt{q})$ each irreducible over  $\mathbb{Q}(\sqrt{7})[X]$ with both of them having roots as $\zeta_{28}$. 
But $[ \mathbb{Q}(\sqrt{7}):\mathbb{Q}(\zeta_{28})]=4$. Passing through completion we have $[ \mathbb{Q}_7(\sqrt{7}):\mathbb{Q}_7(\zeta_{28})]=6$ with root of $F_1(X/ \sqrt{q})$ and $F_2(X/ \sqrt{q})$ as $\zeta{28}$ and since $\deg F_i=6$,  $F_i(X/ \sqrt{q})$ are irreducible over  $\mathbb{Q}_7(\sqrt{7})$. $P(X)=F_1(X)F_2(X)$. But $F_i$  have  coefficients from $\mathbb{Q}_7(\sqrt{7}) / \mathbb{Q}_7$. Hence $P(X)$ is irreducible over $\mathbb{Q}_7$. This implies $\dim A=6$.

\item $P(X)= X^{12} \pm \sqrt{13q}X^{11}+7qX^{10} \pm 3q \sqrt{13q}X^9+15q^2X^8 \pm 5q^2 \sqrt{13q}X^7+19q^3X^6 \pm 5q^3 \sqrt{13q}X^5+15q^4X^4 \pm 3q^4 \sqrt{13q}X^3+7q^5X^2 \pm q^5 \sqrt{13q}X+q^6$ where $p=13$ is irreducible mod 2 hence is irreducible over $\mathbb{Q}$. Using the same argument as above with roots of 
$F_1$ and $F_2$ as $\zeta_{13}$ and $\zeta_{26}$ we get $\dim A =6$.
\item If $P(X)=X^{12}+q^3X^6+q^6$, substitute  $y=qX^2$. Then we get $P(X)=q^6(y^6+y^3+1)$. The polynomial $y^6+y^3+1$ is 9th cyclotomic polynomial.
\begin{enumerate}
\item If $p \equiv 1 \mod 9 $ then $y^6+y^3+1=\displaystyle\prod_{i=1}^{6}(y-\alpha_i)$ over  $\mathbb{Q}_p$ with $v_p(\alpha_i)=0$. Since $n$ is odd we get\\
 $P(X)=\displaystyle\prod_{i=1}^{6}(X^2-\alpha_i q)$. Hence we have 6 invariants with $inv_{\mathfrak{p}_i} (End_k (A)\otimes \mathbb{Q}) \equiv 0 \mod \mathbb{Z}$ which shows the $\dim A= 6$.
\item  If $p \equiv 4,7 \mod 9$  then $y^6+y^3+1=\displaystyle\prod_{i=1}^{2}(y^3+\gamma_i y^2+\beta_i y+\alpha_i)$ with $v_p(\alpha_i)=0$. Since $n$ is odd we get\\
 $P(X)=\displaystyle\prod_{i=1}^{2}(X^6+\gamma_i X^4+\beta_i X^2+\alpha_i)$. Hence we have 2 invariants with $inv_{\mathfrak{p}_i} (End_k (A)\otimes \mathbb{Q}) \equiv 0 \mod \mathbb{Z}$ which shows the $\dim A= 6$.
\item If $p \equiv 2, 5  \mod 9$ then $y^6+y^3+1$ hence $P(X)$ is irreducible over $\mathbb{Q}_p$,  we have 1 invariant with $inv_{\mathfrak{p}_i} (End_k (A)\otimes \mathbb{Q}) \equiv 0 \mod \mathbb{Z}$ which shows the $\dim A= 6$.
\item If $p \equiv 8 \mod 9$  then $y^6+y^3+1=\displaystyle\prod_{i=1}^{3}(y^2+ \beta_i y+\alpha_i)$ with $v_p(\alpha_i)=0$. Since $n$ is odd we get\\
$P(X)=\displaystyle\prod_{i=1}^{3}(X^4+ \beta_i X^2+\alpha_i)$, hence we have 3 invariants with  $inv_{\mathfrak{p}_i} (End_k (A)\otimes \mathbb{Q}) \equiv 0 \mod \mathbb{Z}$ which shows the $\dim A= 6$.
\end{enumerate} 
\item If $P(X)=X^{12}-q^3X^6+q^6$ with $p \neq 3$, substitute  $y=qX^2$. Then we get $P(X)=q^6(y^6-y^3+1)$. The polynomial $y^6-y^3+1$ is 18th cyclotomic polynomial.
\begin{enumerate}
\item If $p \equiv 1 \mod 18 $ then $y^6-y^3+1=\displaystyle\prod_{i=1}^{6}(y-\alpha_i)$ over  $\mathbb{Q}_p$ with $v_p(\alpha_i)=0$. Since $n$ is odd we get\\
 $P(X)=\displaystyle\prod_{i=1}^{6}(X^2-\alpha_i q)$. Hence we have 6 invariants with $inv_{\mathfrak{p}_i} (End_k (A)\otimes \mathbb{Q}) \equiv 0 \mod \mathbb{Z}$ which shows the $\dim A= 6$.
\item  If $p \equiv 7, 13 \mod 18$  then $y^6-y^3+1=\displaystyle\prod_{i=1}^{2}(y^3+\gamma_i y^2+\beta_i y+\alpha_i)$ with $v_p(\alpha_i)=0$. Since $n$ is odd we get\\
 $P(X)=\displaystyle\prod_{i=1}^{2}(X^6+\gamma_i X^4+\beta_i X^2+\alpha_i)$. Hence we have 2 invariants with $inv_{\mathfrak{p}_i} (End_k (A)\otimes \mathbb{Q}) \equiv 0 \mod \mathbb{Z}$ which shows the $\dim A= 6$.
\item If $p \equiv  5, 11  \mod 18$ then $y^6-y^3+1$ hence $P(X)$ is irreducible over $\mathbb{Q}_p$,  we have 1 invariant with $inv_{\mathfrak{p}_i} (End_k (A)\otimes \mathbb{Q}) \equiv 0 \mod \mathbb{Z}$ which shows the $\dim A= 6$.
\item If $p \equiv 17 \mod 9$  then $y^6-y^3+1=\displaystyle\prod_{i=1}^{3}(y^2+ \beta_i y+\alpha_i)$ with $v_p(\alpha_i)=0$. Since $n$ is odd we get\\
$P(X)=\displaystyle\prod_{i=1}^{3}(X^4+ \beta_i X^2+\alpha_i)$, hence we have 3 invariants with  $inv_{\mathfrak{p}_i} (End_k (A)\otimes \mathbb{Q}) \equiv 0 \mod \mathbb{Z}$ which shows the $\dim A= 6$.

\end{enumerate} 
\item If $P(X)= X^{12} - qX^{10}+q^2X^8 - q^3X^6+q^4X^4 - q^5X^2+q^6$ with $p \neq 7$, substitute  $y=qX^2$. Then we get $P(X)=q^6(y^6-y^5+y^4-y^3+y^2-y+1)$. The polynomial $y^6+y^5+y^4-y^3+y^2+y+1$ is 14th cyclotomic polynomial.
\begin{enumerate}
\item If $p \equiv 1 \mod 14 $ then $y^6-y^5+y^4-y^3+y^2-y+1=\displaystyle\prod_{i=1}^{6}(y-\alpha_i)$ over  $\mathbb{Q}_p$ with $v_p(\alpha_i)=0$. Since $n$ is odd we get\\
 $P(X)=\displaystyle\prod_{i=1}^{6}(X^2-\alpha_i q)$. Hence we have 6 invariants with $inv_{\mathfrak{p}_i} (End_k (A)\otimes \mathbb{Q}) \equiv 0 \mod \mathbb{Z}$ which shows the $\dim A= 6$.
\item  If $p \equiv 11,9 \mod 14$  then $y^6-y^5+y^4-y^3+y^2-y+1=\displaystyle\prod_{i=1}^{2}(y^3+\gamma_i y^2+\beta_i y+\alpha_i)$ with $v_p(\alpha_i)=0$. Since $n$ is odd we get\\
 $P(X)=\displaystyle\prod_{i=1}^{2}(X^6+\gamma_i X^4+\beta_i X^2+\alpha_i)$. Hence we have 2 invariants with $inv_{\mathfrak{p}_i} (End_k (A)\otimes \mathbb{Q}) \equiv 0 \mod \mathbb{Z}$ which shows the $\dim A= 6$.
\item If $p \equiv 3, 5  \mod 14$ then $y^6+y^3+1$ hence $P(X)$ is irreducible over $\mathbb{Q}_p$,  we have 1 invariant with $inv_{\mathfrak{p}_i} (End_k (A)\otimes \mathbb{Q}) \equiv 0 \mod \mathbb{Z}$ which shows the $\dim A= 6$.
\item If $p \equiv 13 \mod 14$  then $y^6-y^5+y^4-y^3+y^2-y+1=\displaystyle\prod_{i=1}^{3}(y^2+ \beta_i y+\alpha_i)$ with $v_p(\alpha_i)=0$. Since $n$ is odd we get\\
$P(X)=\displaystyle\prod_{i=1}^{3}(X^4+ \beta_i X^2+\alpha_i)$, hence we have 3 invariants with  $inv_{\mathfrak{p}_i} (End_k (A)\otimes \mathbb{Q}) \equiv 0 \mod \mathbb{Z}$ which shows the $\dim A= 6$.
\end{enumerate} 
\item If $P(X)= X^{12} + qX^{10}+q^2X^8 + q^3X^6+q^4X^4 + q^5X^2+q^6$ , substitute  $y=qX^2$. Then we get $P(X)=q^6(y^6+y^5+y^4+y^3+y^2+y+1)$. The polynomial $y^6+y^5+y^4+y^3+y^2+y+1$ is the 7th cyclotomic polynomial. Using same arguments above, we get $\dim A=6$
\end{enumerate}
\end{proof}

\section{Dimension 7}
The characteristic polynomial of Frobenius of an abelian variety of dimension $7$ is given by

$P(X)=X^{14}+a_{1}X^{13}+a_{2}X^{12}+a_{3}X^{11}+a_{4}X^{10}+a_5X^9+a_6X^8+a_7X^7+qa_6X^6+q^2a_5X^5+q^3a_4X^4+q^4a_{3}X^3+
q^5a_{2}X^2+q^6a_{1}X+q^7.$

If $P(X)$ is irreducible with $a_{2i+1} \neq 0$ then we have following cases.\\
\emph{Case 1.} If $G(x)$ is irreducible as in theorem \ref{S2} then $H(X)$  is a  polynomial of even degree terms only whose roots are mth roots of unity where $\phi=4g= 28$ which implies $m \in \{29, 58\}$. Each of which has following factorization.
\begin{enumerate}
\item $x^{29}-1=( x-1 )(1+x+{x}^{28}+{x}^{27}+{x}^{26}+{x}^{25}+{x}^{24}+{x}^{23}+{x}^{22}+{x}^{21}+{x}^{20}+{x}^{19}+{x}^{18}+{x}^{17}+{x}^{16}+{x}^{15}
    +{x}^{14}+{x}^{13}+{x}^{12}+{x}^{11}+{x}^{10}+{x}^{9}+{x}^{8}+{x}^{7}+{x}^{6}+{x}^{5}+{x}^{4}+{x}^{3}+{x}^{2})$
\item $x^{58}-1=( x-1) (1+x+{x}^{28}+{x}^{27}+{x}^{26}+{x}^{25}+{x}^{24}+{x}^{23}+{x}^{22}+{x}^{21}+{x}^{20}+{x}^{19}+{x}^{18}
    +{x}^{17}+{x}^{16}+{x}^{15}+{x}^{14}+{x}^{13}+{x}^{12}+{x}^{11}+{x}^{10}+{x}^{9}
+{x}^{8}+{x}^{7}+{x}^{6}+{x}^{5}+{x}^{4}+{x}^{3}+{x}^{2})(1+x)  (1-x+{x}^{28}-{x}^{27}+{x}^{26}-{x}^{25}+{x}^{24}-{x}^{23}+{x}^{22}-{x}^{21}+{x}^{20}-{x}^{19}+{x}^{18}-{x}^{17}+{x}^{16}-{x}^{15}
+{x}^{14}-{x}^{13}+{x}^{12}-{x}^{11}+{x}^{10}-{x}^{9}+{x}^{8}-{x}^{7}+{x}^{6}-{x}^{5}+{x}^{4}-{x}^{3}+{x}^{2}) $
\end{enumerate}
 As none of these factors above of degree $28$ are with only even degree terms, we see there is no possibility of this form.
\emph{Case 2.} If $G(x)$ is reducible. Then  by \ref{Red} $G(x)=F_1(x).F_2(x)$.
 If $F_i \in \mathbb{Q}({\sqrt{q}})[x]/ \mathbb{Q}[x]$ then its roots are  mth root of unity  where $\phi(m)=2g=14$  which has no solutions.
 If $F_i \in \mathbb{Q}[x]$ then its roots are  mth root of unity  where  $\phi(m)=g=7$ for which there is no solutions.
Therefore this case is not possible.
\subsection{Case $a_{2i+1}=0$} 
If $a_{2i+1}=0$ then  $G(x)$ is irreducible factor of degree $14$ of $x^m-1$ where   $\phi(m)=2g=14$ which has no solutions.
\subsection{$P(X)$  reducible} 
If $P(X)$ is reducible then $P(X)= h(X)^e$ where $e|7$ , $e>1$ and $h(X)\in \mathbb{Z}[X]$. So $e =7$, in which case $h(X)=X^2+aX \pm q$, this is already discussed in dimension $2$, reducible case and is not a possibility.
 
Hence we have,
\begin{theorem}
There is no simple supersingular abelian variety of dimension $7$.
\end{theorem}

\section{Jacobians of Supersingular Curves}

An important question is whether the above actually occur as the characteristic polynomial of the Frobenius of Jacobian of a curve $C$. Work is in progess on this question. 
For genus $1, ~2$ all of the polynomials listed occur, see \cite{Waterhouse},  respectively. For genus $4$ the following occur  as  Jacobians of hyperelliptic curves of genus $4$ over $\mathbb{F}_q$ where $q=2^n,~ n$ odd. An example of each is given here for $q=2^5$ where $\alpha$ is a primitive root.

The procedure above can be extended to any genus. 

Work is in progress for $n$ even. 

\bigskip

\begin{tabular}{|l|l|} 
\hline
$P(X)$ & Curve\\ \hline
$X^8+\sqrt{2q}X^7+qX^6-q^2X^4+q^3X^2 + q^3\sqrt{2q}X+q^4$ & $ y^2 + y=x^9 + \alpha^2x^5 + \alpha^9x^3 $\\ \hline
 $X^8-\sqrt{2q}X^7+qX^6-q^2X^4+q^3X^2 - q^3\sqrt{2q}X+q^4$& $ y^2 + y=x^9 + \alpha^2x^5 + \alpha^{25}x^3 $\\ \hline
 $ X^8+q^4$ & $y^2 + y=x^9 + x^5 + \alpha^3x^3  $\\ \hline
$X^8+ qX^6+q^2X^4+q^3X^2+q^4$ &$ y^2 + y=x^9 + x^5 + \alpha x^3$\\ \hline
\end{tabular}\\
\\
\\
\textbf{Acknowledgement}\\
The authors would to thank Kevin Hutchinson, Hans Georg Ruck and Christophe Ritzenthaler for valuable discussions.

\end{document}